\documentclass[]{imsart}
\usepackage{amsfonts}
\usepackage{mathrsfs}

\RequirePackage[OT1]{fontenc}
\RequirePackage{amsthm,amsmath,amsfonts,amssymb}
\RequirePackage{amsbsy}
\RequirePackage{mathrsfs}
\RequirePackage[numbers]{natbib}
\RequirePackage[colorlinks,citecolor=blue,urlcolor=blue]{hyperref}

\startlocaldefs
\newtheorem{theorem}{\textbf{Theorem}}
\newtheorem{lemma}{\textbf{Lemma}} [section]

\newtheorem{ass}{\textbf{Assumption}}
\newtheorem{remark}{\textbf{Remark}}[section]
\endlocaldefs

\allowdisplaybreaks[4]

\begin{document}

\begin{frontmatter}
\title{ Central Limit Theorems of Local Polynomial Threshold Estimators for Diffusion Processes with Jumps}
\runtitle{Central Limit Theorems of Local Polynomial Threshold Estimators}

\begin{aug}
\author{\fnms{Yuping} \snm{Song}}

\address{ School of Finance and Business, Shanghai Normal
         University, Shanghai,   200234,  P.R.C.}

\author{\fnms{Hanchao} \snm{Wang}\thanksref{t2}\ead[label=e3]{wanghanchao@sdu.edu.cn}
\ead[label=u1,url]{www.foo.com}}

\address{Zhongtai Securities Institute for Financial Studies,  Shandong University,  Jinan,  250100, P.R.C.\\
\printead{e3} }

\thankstext{t1}{This research work is support by the National Natural Science Foundation of China (No. 11371317, 11171303)
and the Fundamental Research Fund of Shandong University (No.
2016GN019). }
\thankstext{t2}{Corresponding author}

\runauthor{Y. Song et al.}

\end{aug}

\begin{abstract}
Central limit theorems play an important role in the study of
statistical inference for stochastic processes. However, when the
nonparametric local polynomial threshold estimator, especially local
linear case, is employed to estimate the diffusion coefficients of
diffusion processes, the adaptive and predictable structure of the
estimator conditionally on the $\sigma-$field generated by diffusion
processes is destroyed, the classical central limit theorem for
martingale difference sequences can not work. In this paper, we proved the central limit theorems of local polynomial
threshold estimators for the volatility function in diffusion
processes with jumps. We
believe that our proof for local polynomial threshold estimators
provides a new method in this fields, especially local linear case.
\end{abstract}



\end{frontmatter}

\section{Introduction}

Volatility is an important feature of financial markets, which is
directly related to market uncertainty and risk. It is not only an
effective indicator of quality and efficiency for financial market,
but also a core variable for portfolio theory, asset price modeling,
arbitrage price modeling and option price formula. Hence, how to
effectively describe the dynamic behavior of volatility for
financial market has always been the core problem.

Estimating the price volatility of financial assets exactly is
fundamental for the financial risk management, which has long been
the focus of the theoretical study and empirical application such as
risk management, asset pricing, proprietary trading and portfolio
managements. In finance, ``tick data'' are recorded at every
transaction time (sampled at very high frequency), so we do get huge
amounts of data on the prices or return rates of various assets and
so on. In this context, it gives a new challenge to study the
estimators for the process, which characterizes the prices or
returns of various assets and so on. With the development of
financial statistical methods, using real-time transaction data to
estimate asset return volatility has become a hot topic. In high
frequency context, more and more statisticians and economists are
interested in the nonparametric inference for diffusion coefficients
of stochastic processes which characterize the dynamics of option
prices, interest rates, exchange rates and inter alia.

In this paper, we assume that all processes are defined on a
filtered probability space $(\Omega, (\mathscr{F}_{t})_{t\in[0,T]},
\mathscr{F}, P)$, satisfying the usual conditions (Jacod and
Shiryaev \cite{j}). A diffusion processes can be represented by the
solution of following stochastic differential equation:
\begin{equation}
\label{dif} dX_{t}=\mu(X_{t})dt+\sigma(X_{t})dW_{t}\end{equation}
where $W_{t}$ is a standard Brownian motion. Diffusion processes
play an important role in the study of mathematical financial.
Especially, many models in economics and finance, like those for an
interest rate or an asset price, involve
 diffusion processes. Assuming that the process (\ref{dif}) is observed at $n+1$ discrete
time observations $\{X_{0}, X_{t_{1}}, ... , X_{t_{n-1}},
X_{t_{n}}\}$ with $t_{i} - t_{i-1} = \delta$, ~for $i = 1, ... , n$
and $t_{n} = T$, based on the infinitesimal conditional moment
restriction
\begin{equation}
\label{d2}
\lim_{\delta\rightarrow0}E[\frac{(X_{t+\delta}-X_{t})^2}{\delta}|X_{t}=x]=\sigma^2(x),\end{equation}
the nonparametric estimators for volatility function $\sigma^2(x)$
can be constructed by the nonparametric regression method. The
Nadaraya-Watson estimators is a natural choice, we can estimate
$\sigma^2(x)$ through
\begin{equation}
\label{n2}
\breve{\sigma}^2_n(x)=\frac{\sum_{i=1}^{n}K(\frac{X_{t_{i-1}}-x}{h_n})\frac{(X_{t_{i}}-X_{t_{i-1}})^2}{\delta}}
{\sum_{i=1}^{n}K(\frac{X_{t_{i-1}}-x}{h_n})},
\end{equation}
where $K(\cdot)$ is a kernel function, and $h_n$ is bandwidth. Bandi
and Phillips \cite{bp}  obtained the central limit theorems for
$\breve{\sigma}^2_n(x)-\sigma^2(x)$ by Knight's embedding theorem (
Revuz and Yor \cite{ry}). There are many methods to improve the
statistical behaviors of $\breve{\sigma}^2_n(x)$.  Local polynomial
smoothing method is a popular method for improving Nadaraya-Watson
estimator. Fan and Zhang (\cite{fz}) first estimated $\sigma^2(x)$
through local polynomial method, however, the asymptotic normality
was not obtained by them (they only computed the bias and variance
for the estimator of the diffusion coefficient).

Recently, diffusion processes with jumps as an intension of
continuous-time ones have been studied by more and more
statisticians and economists since the financial phenomena can be
better characterized (see Johannes \cite{joh}, A\"it-Sahalia and
Jacod \cite{aj}, Bandi and Nguyen \cite{bn}). It is natural to
consider the following model:
\begin{equation}
\label{om}
dX_{t}=\mu(X_{t-})dt+\sigma(X_{t-})dW_{t}+dJ_{t},\end{equation}
where $J_{t}$ is a pure jump semimartingale. The jumps $J_{t} =
J_{1t} + \tilde{J}_{2t}$ consist of large and infrequent jump
component $J_{1t}$ (finite activity) as well as small and frequent
jump component with finite variation $\tilde{J}_{2t}$ (infinite
activity). Ordinarily, $J_{1t}$ is assumed compound Poisson
processes and $\tilde{J}_{2t}$ is assumed to be L\'{e}vy. One can
refer to Bandi and Nguyen \cite{bn} for doubly stochastic compound
Poisson process, Madan \cite{mad} for Variance Gamma process, Carr
et al. \cite{car} for the CGMY model with $Y < 1$, Cont and Tandov
\cite{ct} for $\alpha-$stable or tempered stable process with
$\alpha < 1$. Disentangling the jump component from observations is
essential for risk management, one can refer to Andersen et al.
\cite{and} and Corsi et al. \cite{cor}. In presence of jump
component $J_{1t},$ Bandi and Nguyen \cite{bn}, Johannes \cite{joh}
constructed nonparametric estimation for $\sigma^2(x)$ based on
estimation of infinitesimal moments and provided central limiting
theory.

Under the influence of jumps, especially $\tilde{J}_{2t}$, how to
estimate $\sigma^2(x)$ is an interesting problem. For finite
activity case, Mancini \cite{m1} showed that due to the continuity
modulus of the Brownian motion paths, it is possible to disentangle
in which intervals jumps occur when the interval between two
observations shrinks for $T < \infty$. This property allows one to
identify the jump component asymptotically and remove it from X.
Mancini and Ren\`{o} \cite{mr} showed that this methodology is
robust to enlarging time span ($T \rightarrow \infty$) and to the
presence of infinite activity jumps $\tilde{J}_{2t}$. For more
knowledge of this aspect, one also can refer to Barndorff-Nielsen
and Shephard \cite{bs}, Mancini \cite{m0}, A\"it-Sahalia and  Jacod
\cite{aj} for alternative approaches to disentangle jumps from
diffusion based on power and multipower variation.

Mancini and Ren\`{o} \cite{mr} combined the Nadaraya-Watson
estimator and threshold  method to eliminate the impact of jumps.
They estimated $\sigma^2(x)$ through
\begin{equation}
\label{t2}
\bar{\sigma}^2_n(x)=\frac{\sum_{i=1}^{n}K(\frac{X_{t_{i-1}}-x}{h_n})\frac{(X_{t_{i}}-X_{t_{i-1}})^2}{\delta}I_{\{(\Delta_{i}X)^2
\leq \vartheta(\delta)\}}}
{\sum_{i=1}^{n}K(\frac{X_{t_{i-1}}-x}{h_n})}.
\end{equation}
In the context of nonparametric estimator with finite-dimensional
auxiliary variables, local polynomial smoothing become the ``golden
standard'', see Fan \cite{fan92}, Wand and Jones \cite{wj}. The
local polynomial estimator is known to share the simplicity and
consistency of the kernel estimators such as Nadaraya-Watson or
Gasser-M\"{u}ller estimators. Moreover, when the convergence rates
are concerned, local polynomial estimator possesses simple bias
representation and corrects the boundary bias automatically.
However, when the nonparametric local polynomial threshold estimator
is employed to estimate volatility function $\sigma^{2}(x)$ for
better bias properties instead of Nadaraya Watson estimator, the
adaptive and predictable structure of estimator is destroyed, so the
classical central limit theorem for martingale difference sequences
can not work. In this paper, we will discuss this problem and prove
the central limit theorem for local linear threshold estimator for
the diffusion coefficient $\sigma^2(x).$

The remainder of this paper is organized as follows. Stable
convergence and its property is shown in section 2. Section 3
introduces our model, local polynomial threshold estimator and main
results. The proofs of the results will be collected in section 4.

\section{Stable convergence and its property}

In this section, firstly, we will define the stable convergence in
law and mention its property, secondly, we will show limit theorem
for partial sums of triangular arrays of random variables, one can
refer to Jacod and Shiryaev \cite{j} or Jacod \cite{jac} for more
details.

\noindent {\textbf{1) Stable convergence in law.}}

This notation was firstly introduced by R\'{e}nyi \cite{ren}, which
in the same reason we need here for the proof, and exposited by
Aldous and Eagleson \cite{ald}.

A sequence of random variables $Z_{n}$ defined on the probability
space $(\Omega, \mathcal {F}, \mathbb{P}),$ taking their values in
the state space $(E, \mathcal {E}),$ assumed to be Polish. We say
that $Z_{n}$ stably converges in law if there is a probability
measure $\eta$ on the product $(\Omega \times E, \mathcal {F} \times
\mathcal {E}),$ such that $\eta(A \times E) = \mathbb{P}(A)$ for all
$A \in \mathcal {F}$ and
\begin{equation}
\label{def1} \mathbb{E}(Y f(Z_{n})) \longrightarrow \int{Y(\omega)
f(x) \eta(d\omega, dx)}
\end{equation}
for all bounded continuous functions $f$ on $E$ and bounded random
variables $Y$ on $(\Omega, \mathcal {F}).$

Take $\tilde{\Omega} = \Omega \times E$, $\tilde{\mathcal {F}} =
\mathcal {F} \times \mathcal {E}$ and endow $(\tilde{\Omega},
\tilde{\mathcal {F}})$ with the probability $\eta,$ and put
$Z(\omega, x) = x,$ on the extension $(\tilde{\Omega},
\tilde{\mathcal {F}}, \tilde{\mathbb{P}})$ of $({\Omega}, {\mathcal
{F}}, {\mathbb{P}})$ with the expectation $\tilde{\mathbb{E}}$ we
have
\begin{equation}
\label{def2} \mathbb{E}(Y f(Z_{n})) \longrightarrow
\tilde{\mathbb{E}}(Y f(Z)),
\end{equation}
then we say that $Z_{n}$ converges stably to $Z,$ denoted by
$\stackrel{\mathcal {S} - \mathcal {L}} \longrightarrow.$

The stable convergence implies the following crucial property, which
is fundamental for the mixed normal distribution with random
variance of the local polynomial estimator, detailed in the proof of
Theorem 1 and 2.

\noindent $\bf{Lemma~2.1.}$~~if $Z_{n} \stackrel{\mathcal {S} -
\mathcal {L}} \longrightarrow Z $ and if $Y_{n}$ and $Y$ are
variables defined on $(\Omega, \mathcal {F}, \mathbb{P})$ and with
values in the same Polish space F, then
\begin{equation}
\label{def3} Y_{n} \stackrel{P} \longrightarrow
Y~~~~~\Rightarrow~~~~~(Y_{n}, ~Z_{n}) \stackrel{\mathcal {S} -
\mathcal {L}} \longrightarrow (Y, ~Z),
\end{equation}
which implies that $Y_{n} \times Z_{n} \stackrel{\mathcal {S} -
\mathcal {L}} \longrightarrow Y \times Z$ through the continuous
function $g(x, y) = x*y.$

\noindent {\textbf{2) Convergence of triangular arrays.}}

In this part, we give the available convergence criteria for stable
convergence of partial sums of triangular arrays, one can refer to
Jacod \cite{jac} (P17-Lemma 4.4).

\noindent $\bf{Lemma~2.2.}$~~{\textbf{[Jacod's stable convergence
theorem]} A sequence of $\mathbb{R}-$valued variables $(\zeta_{n,
i}: i \geq 1)$ defined on the filtered probability space $(\Omega,
\mathcal {F}, (\mathcal {F})_{t \geq 0}, \mathbb{P})$ is $\mathcal
{F}_{i \Delta_{n}}-$measurable for all $n, i.$ Assume there exists a
continuous adapted $\mathbb{R}-$valued process of finite variation
$B_{t}$ and a continuous adapted and increasing process $C_{t}$, for
any $t
> 0,$ we have
\begin{equation}
\label{def4} \sup_{0 \leq s \leq t}\big|\sum_{i =
1}^{[s/\Delta_{n}]}\mathbb{E}\big[\zeta_{n, i} | \mathcal {F}_{(i -
1)\Delta_{n}}\big] - B_{s}\big| \stackrel{p} \longrightarrow 0,
\end{equation}
\begin{equation}
\label{def5} ~~~~~~~~~~~~~~~~~~~~~~~~~\sum_{i =
1}^{[t/\Delta_{n}]}\big(\mathbb{E}\big[\zeta_{n, i}^{2} | \mathcal
{F}_{(i - 1)\Delta_{n}}\big] - \mathbb{E}^{2}\big[\zeta_{n, i} |
\mathcal {F}_{(i - 1)\Delta_{n}}\big]\big) - C_{t} \stackrel{p}
\longrightarrow 0,
\end{equation}
\begin{equation}
\label{def6} \sum_{i = 1}^{[t/\Delta_{n}]}\mathbb{E}\big[\zeta_{n,
i}^{4} | \mathcal {F}_{(i - 1)\Delta_{n}}\big] \stackrel{p}
\longrightarrow 0.
\end{equation}
Assume also
\begin{equation}
\label{def7} \sum_{i = 1}^{[t/\Delta_{n}]}\mathbb{E}\big[\zeta_{n,
i} \Delta_{n}^{i}H | \mathcal {F}_{(i - 1)\Delta_{n}}\big]
\stackrel{p} \longrightarrow 0,
\end{equation}
where either H is one of the components of Wiener process $W$ or is
any bounded martingale orthogonal (in the martingale sense) to $W$
and $\Delta_{n}^{i}H = H_{i\Delta_{n}} - H_{(i - 1)\Delta_{n}}.$

Then the processes $$\sum_{i = 1}^{[t/\Delta_{n}]}\zeta_{n, i}
\stackrel{\mathcal {S} - \mathcal {L}} \longrightarrow B_{t} +
M_{t},$$ where $M_{t}$ is a continuous process defined on an
extension $\big(\widetilde{\Omega}, \widetilde{P},
\widetilde{\mathcal {F}}\big)$ of the filtered probability space
$\big({\Omega}, {P}, {\mathcal {F}}\big)$ and which, conditionally
on the the $\sigma-$filter $\mathcal {F}$, is a centered Gaussian
$\mathbb{R}-$valued process with $\widetilde{E}\big[M_{t}^{2} |
\mathcal {F}\big] = C_{t}.$

\noindent {\textbf{Remark 2.1.}} As Jacod \cite{jac} mentioned that
the key assumption of Lemma 2.2 is that for all $n,~i$ the variable
$\zeta_{n, i}$ is $\mathcal {F}_{i\Delta_{n}}-$measurable. For
Nadaraya-Watson estimator, the triangular arrays of numerator in
(5):
$$\sum_{i=1}^{n}K\left(\frac{X_{t_{i-1}}-x}{h_n}\right)\frac{(X_{t_{i}}-X_{t_{i-1}})^2}{\delta}I_{\{(\Delta_{i}X)^2
\leq \vartheta(\delta)\}}$$ is $\mathcal
{F}_{i\Delta_{n}}-$measurable, so Mancini and Ren\`{o} \cite{mr} can
employ Lemma 2.2 to prove the stable convergence for numerator of
$\bar{\sigma}^2_n(x).$ However, for local linear estimator, the
triangular arrays of numerator in (\ref{c1}):
$$\sum_{i=1}^{n}K\left(\frac{X_{t_{i-1}}-x}{h}\right)\left\{\frac{\delta
S_{n,2}}{h^2}-(\frac{X_{t_{i-1}}-x}{h})\frac{\delta
S_{n,1}}{h}\right\}\frac{(\Delta_{i}X)^{2} I_{\{(\Delta_{i}X)^2 \leq
\vartheta(\delta)\}}}{\delta}$$ with $S_{n,k} =
\frac{1}{h}\sum_{i=1}^{n}K\left(\frac{X_{i\delta}-x}{h}\right)\left(X_{i\delta}-x\right)^k$,
for $k= 1, 2$ is not $\mathcal {F}_{i\Delta_{n}}-$measurable due to
$S_{n,k}$, so we can not directly employ Lemma 2.2 to show the
stable convergence for it. Fortunately, we could deal with the
problem under some techniques with the help of Lemma 2.1, one can
refer to the third part or the detailed proof for some understanding
the methodology.

\section{Setting and Main results}

Recall that  a diffusion process with jumps can be defined by the
following stochastic differential equation (\ref{om}):
$$dX_{t}=\mu(X_{t-})dt+\sigma(X_{t-})dW_{t}+dJ_{t},\qquad t\in[0 , T]$$
where $\mu(x)$ and $\sigma(x)$ are smooth functions,
$W=\{W_{t},t\geq 0\}$ is a standard Brownian motion, where $J_{t}$
is a pure jump semimartingale. The jumps $J_{t} = J_{1t} +
\tilde{J}_{2t}$ consist of large and infrequent jump component
$J_{1t}$ (finite activity) as well as small and frequent jump
component with finite variation $\tilde{J}_{2t}$ (infinite
activity). $J_{1t}$ is a finite activity (FA) pure jump
semimartingale (e.g. driven by a doubly stochastic compound Poisson
process with jump intensity $\lambda( \cdot )$ in $L^1(\Omega \times
[0 , +\infty))$), independent of $\{W_{t},t\ge 0\}$.

Generally, since $J_{1t}$ is any FA pure jump semimartingale ,which
we can write as $$J_{1t} = \int_0^t\int_{\mathscr{R}}{x\cdot
m(dx,du)} =\sum_{l=1}^{N_t}\gamma_l,$$ where $m$ is the jump random
measure of $J_{1t}$, the jump intensity $\lambda( \cdot )$ is a
stochastic process, and $N_t :=\int_0^t\int_{\mathscr{R}}{1\cdot
m(dx,du)}$ is a.s. finite.

$\tilde{J}_{2t}$ is assumed to be a pure jump L\'{e}vy process of
type $$\tilde{J}_{2s} := \int_{0}^{s} \int_{|x| \leq 1} x [m(dt, dx)
- \nu(dx) dt]$$ with $\nu \{|x| \leq 1\} = + \infty,$ where $\nu$ is
the L\'{e}vy measure of $\tilde{J}_{2}.$ Cont and Tandov \cite{ct}
discussed the $Blumenthal-Getoor~index$ for any L\'{e}vy process:
$$\alpha := \inf \left\{\delta \geq 0~:~\int_{|x| \leq 1} |x|^{\delta} \nu(dx) < + \infty \right\},$$
which measure how frenetic the jump activity. Here we only consider
the case $\alpha < 1$, which implies $\tilde{J}_{2}$ has finite
variation, that is, $\sum_{s \leq T} \Delta \tilde{J}_{2s} <
\infty.$ In this case, Protter \cite{pro} showed that there exists
the local time $L_{t}(x)$, which is continuous in $t$ and
c\`{a}dl\`{a}g in $x$, and the occupation time formula keeps true.

As a nonparametric methodology, the local polynomial estimator has
received increasing attention and become a powerful and useful
diagnostic tool for data analysis making use of the observation
information to estimate corresponding functions and its derivatives
without assuming the function form. The estimator is obtained by
locally fitting $p$-th polynomial to the data via weighted least
squares. The procedure of weighted local polynomial regression is
conducted as follows: under some smoothness conditions of the curve
$m(x)$, we can expand $m(x)$ in a neighborhood of the point $x_{0}$
as follows:
\begin{eqnarray*}
m(x) & \approx & m(x_{0}) + m^{'}(x_{0})(x - x_{0}) +
\frac{m^{''}(x_{0})}{2!}(x - x_{0})^{2} + \cdots +
\frac{m^{(p)}(x_{0})}{p~!}(x - x_{0})^{p}\\
& \equiv & \sum_{j=0}^{p}\beta_{j}(x - x_{0})^{j},
\end{eqnarray*}
where $\beta_{j} = \frac{m^{(j)}(x_{0})}{j~!}.$

Thus, the problem of estimating infinite dimensional $m(x)$ is
equivalent to estimating the $p$-dimensional parameter $\beta_{0},
\beta_{0}, \cdots, \beta_{p}.$

When we want to estimate $\sigma^2(x)$ in model (\ref{om}) from the
discrete time observations $\{X_{0}, X_{t_{1}}, ... , X_{t_{n-1}},
X_{t_{n}}\}$, with $t_{i} - t_{i-1} = \delta$, we can consider a
weighted local polynomial regression through the threshold  method
to eliminate the impact of jumps:
\begin{equation}
\label{q2} \arg \min_{\beta_{0},\beta_{1},\cdots,\beta_{p}}
\sum_{i=0}^{n-1} \Big\{Y_{t_i} - \sum_{j=0}^{p}\beta_{j}(X_{t_i} -
x)^{j}\Big\}^{2}K_{h_{n}}(X_{t_i} - x),\end{equation} where $Y_{t_i}
= \frac{(X_{t_{i+1}} - X_{t_i})^2}{\delta}I_{\{(\Delta_{i}X)^2 \leq
\vartheta(\delta)\}}$ and
$K_{h_{n}}(\cdot)=\frac{1}{h_{n}}K(\frac{\cdot}{h_{n}})$ is kernel
function with $h_{n}$ the bandwidth, $\vartheta(\delta)$ is a
threshold function.

Under the algebra calculus (one can refer to Fan and Gijbels
\cite{fg}), we obtain the solution to this minimization problem
(\ref{q2}) is
\begin{equation}
\label{s1} \hat{\beta} := (\hat{\beta}_{0}, \hat{\beta}_{1}, \cdot
\cdot \cdot, \hat{\beta}_{p})^{T} = S_{n}^{- 1} Q_{n},
\end{equation}
with
$$S_{n}~=~\begin{pmatrix}
S_{n, 0} & \cdot \cdot \cdot & S_{n, p} \\
\cdot \cdot \cdot & ~~~ & \cdot \cdot \cdot\\
S_{n, p} & \cdot \cdot \cdot & S_{n, 2p}
\end{pmatrix}, \qquad
Q_{n}~=~\begin{pmatrix}
Q_{n, 0} \\
\cdot \cdot \cdot\\
Q_{n, p}
\end{pmatrix},$$ where $$S_{n, k} =
\frac{1}{h}\sum_{i=1}^{n}K\left(\frac{X_{i\delta}-x}{h}\right)\left(X_{i\delta}-x\right)^k$$
and $$Q_{n, k} =
\frac{1}{h}\sum_{i=1}^{n}K\left(\frac{X_{i\delta}-x}{h}\right)\left(X_{i\delta}-x\right)^k
\frac{(X_{t_{i+1}} - X_{t_i})^2}{\delta}I_{\{(\Delta_{i}X)^2 \leq
\vartheta(\delta)\}}.$$

\noindent As Fan \cite{fan92} showed, since this methodology is
mainly conducted by means of locally fitting $p$-th polynomial, the
degree is not allowed higher, usually $\tau + 1$ and rarely $\tau +
3,$ where $\tau$ is the degree of unknown function we need to
estimate in $\hat{\beta}.$ What we are interested in estimating
$\sigma^2(x)$ is $\hat{\beta}_{0},$ that is $\tau = 0,$  it is
reasonable for us to discuss $p = 1$ in this paper, which is the
local linear estimator.

In fact, we can write the solutions $\hat{\beta}_{0}$ of (\ref{q2})
with $p = 1$ for (\ref{s1}), that is,
\begin{equation}
\label{c1} \hat\sigma^2_{n}(x) =
\frac{\sum_{i=1}^{n}K(\frac{X_{t_{i-1}}-x}{h})\{\frac{\delta
S_{n,2}}{h^2}-(\frac{X_{t_{i-1}}-x}{h})\frac{\delta
S_{n,1}}{h}\}\frac{(\Delta_{i}X)^2 I_{\{(\Delta_{i}X)^2 \leq
\vartheta(\delta)\}}}{\delta}}{\sum_{i=1}^{n}K(\frac{X_{t_{i-1}}-x}{h})\{\frac{\delta
S_{n,2}}{h^2}-(\frac{X_{t_{i-1}}-x}{h})\frac{\delta
S_{n,1}}{h}\}}\end{equation}
 where
$\Delta_{i}X=X_{t_{i}}-X_{t_{i-1}}$, $S_{n,k} =
\frac{1}{h}\sum_{i=1}^{n}K(\frac{X_{i\delta}-x}{h})(X_{i\delta} -
x)^k$, \ for $k= 1, 2$.

In fact, there are many papers on local linear estimator in
regression analysis and time series analysis, more details can be
found in Fan and Gijbels \cite{fg}. The primary purpose of the
present paper is to establish central limit theorems for
$\hat\sigma^2_{n}(x)-\sigma^2(x).$

The triangular arrays of numerator of local linear estimator in
(\ref{c1}) is
\begin{eqnarray*}
\hat\sigma^2_{n}(x)^{Num} & = &
\sum_{i=1}^{n}K\left(\frac{X_{t_{i-1}}-x}{h}\right)\left\{\frac{\delta
S_{n,2}}{h^2}-(\frac{X_{t_{i-1}}-x}{h})\frac{\delta
S_{n,1}}{h}\right\}\frac{(\Delta_{i}X)^{2} I_{\{(\Delta_{i}X)^2 \leq
\vartheta(\delta)\}}}{\delta}\\
& = & \begin{pmatrix} \frac{\delta S_{n,2}}{h^2}, & \frac{\delta
S_{n,1}}{h} \end{pmatrix} \begin{pmatrix}
\sum_{i=1}^{n}K\left(\frac{X_{t_{i-1}}-x}{h}\right) \\
\sum_{i=1}^{n}K\left(\frac{X_{t_{i-1}}-x}{h}\right)\left(\frac{X_{t_{i-1}}-x}{h}\right)\end{pmatrix}
\frac{(\Delta_{i}X)^{2} I_{\{(\Delta_{i}X)^2 \leq
\vartheta(\delta)\}}}{\delta}.
\end{eqnarray*}
We have shown in remark 2.1 that we can not directly employ Lemma
2.2 to show the stable convergence for the numerator. With the help
of lemma 4.3, we obtain $\frac{\delta S_{n,2}}{h^2}
\stackrel{\mathrm{a.s.}}{\longrightarrow}~\frac{K_{1}^{2}L_{X}(T,x)}{\sigma^2(x)}$
and $\frac{\delta S_{n,1}}{h}
\stackrel{\mathrm{a.s.}}{\longrightarrow}~\frac{K_{1}^{1}L_{X}(T,x)}{\sigma^2(x)}.$
Hence,
\begin{eqnarray*}
\frac{\hat\sigma^2_{n}(x)^{Num}}{L_{X}(T,x)} &
\stackrel{\mathrm{a.s.}}{\longrightarrow} &
\begin{pmatrix} \frac{K_{1}^{2}}{\sigma^2(x)}, & \frac{K_{1}^{1}}{\sigma^2(x)} \end{pmatrix}
\begin{pmatrix}
\sum_{i=1}^{n}K\left(\frac{X_{t_{i-1}}-x}{h}\right) \\
\sum_{i=1}^{n}K\left(\frac{X_{t_{i-1}}-x}{h}\right)\left(\frac{X_{t_{i-1}}-x}{h}\right)\end{pmatrix}
\frac{(\Delta_{i}X)^{2} I_{\{(\Delta_{i}X)^2 \leq
\vartheta(\delta)\}}}{\delta}\\
& = & \sum_{i=1}^{n}
K\left(\frac{X_{t_{i-1}}-x}{h}\right)\left\{\frac{K_{1}^{2}}{\sigma^2(x)}
- (\frac{X_{t_{i-1}}-x}{h})\frac{K_{1}^{1}}{\sigma^2(x)}
\right\}\frac{(\Delta_{i}X)^{2} I_{\{(\Delta_{i}X)^2 \leq
\vartheta(\delta)\}}}{\delta}\\
& := & \sum_{i=1}^{n} \zeta_{n, i}.
\end{eqnarray*}
Obviously, the triangular arrays $\zeta_{n, i}$ is $\mathcal
{F}_{i\Delta_{n}}-$measurable, so we can utilize lemma 2.2 to prove
the stable convergence in law for $\sum_{i=1}^{n} \zeta_{n, i}.$
From lemma 4.3, we know that
$\frac{1}{h}\sum_{i=1}^{n}K\big(\frac{X_{t_{i-1}}-x}{h}\big)\delta~
\stackrel{\mathrm{a.s.}}{\longrightarrow}~\frac{L_{X}(T,x)}{\sigma^2(x)}$,
which implies that we can prove the stable convergence in law for
$\hat\sigma^2_{n}(x)^{Num}$ by means of the property as lemma 2.1,
more details can be sketched in the proof of Theorem 1.

Assume that $\mathfrak{D}=(l,u)$ with $-\infty\le l < u\leq \infty$
is the range of the process $X$. We will use notation
``$\stackrel{p}{\rightarrow}$'' to denote ``convergence in
probability'', ``$\stackrel{a.s.}{\rightarrow}$'' to denote
``convergence almost surely'', ``$\Rightarrow$'' to denote
``convergence in distribution'' and ``$\stackrel{\mathcal {S} -
\mathcal {L}} \longrightarrow$'' to denote ``stable convergence in
law''. We impose the following assumptions throughout the paper.

\begin{ass}\label{amu}
For model (\ref{om}), the coefficients $\mu_{t}$ and $\sigma_{t}$
are progressively measurable process with c\`{a}dl\`{a}g paths and
the following polynomial growth:

(i) For each $n \in \mathbb{N},$ there exists a positive  constant
$L_{n}$ such that for any $|x| \leq n, |y| \leq n$,
$$|\mu(x) - \mu(y)| + |\sigma(x) - \sigma(y)| \leq L_{n}|x - y|,$$

(ii)There exists a positive constant  C, such that for all $x \in
\mathbb{R}$,
$$|\mu(x)| + |\sigma(x)| \leq C (1 + |x|),$$

(iii) $\sigma^{2}$ is strictly positive and $\sigma^{'}$ is bounded.
\end{ass}

\begin{remark}
This assumption (i) and (ii) guarantees the existence and uniqueness
of a strong solution to $X$ in Eq.(\ref{om}) on our filtered
probability space

\noindent $(\Omega, (\mathscr{F}_{t})_{t\in[0,T]}, \mathscr{F}, P)$,
which is adapted with c\`{a}dl\`{a}g paths on $[0, T]$, see Ikeda
and Watanabe \cite{iw} for more details.
\end{remark}

\begin{ass}\label{ak}
The kernel function K $(\cdot) : \mathbb{R}^{+} \rightarrow
\mathbb{R}^{+}$ is a continuous differentiable and bounded density
function with bounded compact support, such that $\int_{0}^{\infty}
K (u)du = 1$ and $\int_{0}^{\infty}{|K^{'}(u)|^{2}}du < \infty$.
Denote $K_{i}^{j} = \int_{0}^{\infty}{u^{j}K^{i}(u)} du < \infty.$
\end{ass}

\begin{remark}
The one-sided and asymmetric kernel function is mentioned in
assumption 4 (ii) of Bandi and Nguyen \cite{bn}. Fan and Zhang
\cite{fz} proposed that the one-sided kernel function will make
prediction easier, such as the Epanechnikov kernel $K(u) =
\frac{3}{4}(1 - u^{2}) I_{(u < 0)}.$
\end{remark}

\begin{ass}\label{ah}
A bandwidth parameter is a sequence of real number $h_{n}$ such that
as $n \longrightarrow \infty,$ we have $h_{n} \longrightarrow 0,$
$\frac{\delta{ln(\frac{1}{\delta})}}{h^{2}} \rightarrow 0.$
\end{ass}

For model (\ref{om}), under the assumptions, we build the
corresponding theorems of local linear threshold estimators
(\ref{c1}) for different jump cases.

\noindent $\bf{Finite~Activity~Jumps~(FA~case)}:$

\noindent In (\ref{om}), if we assume that $J_{t} = J_{1, t} =
\sum_{l=1}^{N_{t}}\gamma_{l}$, where $N_{t}$ is a doubly stochastic
Poisson process with an intensity process $\lambda(X_{t-})$, we have
the following result.
\begin{theorem}\label{th1}
Under Assumptions \ref{amu}, \ref{ak}, \ref{ah} and we also assume that\\
(1) as $\delta \rightarrow 0$ both the threshold function
$\vartheta(\delta)$ and $\frac{\delta
ln(\frac{1}{\delta})}{\vartheta(\delta)}$ tend to 0;\\
(2) $\frac{h_{n}^{5}}{\delta_{n, T}} = O_{p}(1),$ then we can obtain
$$\sqrt{\frac{h_{n}}{\delta_{n, T}}}\left(\hat\sigma^2_{n}(x) - \sigma^2(x) - \frac{1}{2} \frac{(\sigma^{2})^{''}(x)[(K_{1}^{2})^{2} -
K_{1}^{1}K_{1}^{3}]}{K_{1}^{2} - (K_{1}^{1})^{2}}  \cdot h^{2}
\right) \stackrel{\mathcal {S} - \mathcal {L}} \longrightarrow M
\mathscr{N} \left(0 , 2 \frac{\sigma^{6}(x)}{L_{X}(T,x)} \cdot
V_{x}\right),$$ where $\delta_{n, T} = \frac{T}{n}$, $V_{x} =
\frac{K_{2}^{0}\cdot\big(K_{1}^{2}\big)^{2} +
K_{2}^{2}\cdot\big(K_{1}^{1}\big)^{2} - 2K_{2}^{1}\cdot
K_{1}^{2}\cdot K_{1}^{1}}{(K_{1}^{2} - (K_{1}^{1})^{2})^{2}}$ and
$M\mathscr{N}(0, U^2)$ is a random variable having a mixed normal
law with the characteristic function $\phi(u) =
E[e^{-\frac{U^2u^2}{2}}].$
\end{theorem}

\noindent {\textbf{Corollary~1.}} Under Assumptions \ref{amu}, \ref{ak}, \ref{ah} and we also assume that\\
(1) as $\delta \rightarrow 0$ both the threshold function
$\vartheta(\delta)$ and $\frac{\delta
ln(\frac{1}{\delta})}{\vartheta(\delta)}$ tend to 0;\\
(2) $\frac{h_{n}^5 \hat{L}_{X}(T,x)}{\delta_{n, T}} = O_{p}(1),$
then we can obtain
$$\sqrt{\frac{h_{n} \hat{L}_{X}(T,x) }{\delta_{n, T}}}\left(\hat\sigma^2_{n}(x) - \sigma^2(x) - \frac{1}{2} \frac{(\sigma^{2})^{''}(x)[(K_{1}^{2})^{2} -
K_{1}^{1}K_{1}^{3}]}{K_{1}^{2} - (K_{1}^{1})^{2}}  \cdot h^{2}
\right) \Rightarrow \mathscr{N} \left(0 , 2 \sigma^{4}(x) \cdot
V_{x}\right),$$ where $\hat{L}_{X}(T,x) =
\frac{1}{h}\sum_{i=1}^{n}K\big(\frac{X_{t_{i-1}}-x}{h}\big) \delta.$

\begin{remark}
According to Lemma 4.3, we know that $\hat{L}_{X}(T,x)
\stackrel{a.s.} \longrightarrow \frac{L_{X}(T,x)}{\sigma^{2}(x)},$
so we can deduce corollary 1 by means of lemma 2.1 easily with the
property that the stable convergence implies convergence in
distribution.
\end{remark}

\noindent $\bf{Infinite~Activity~Jumps~(IA~case)}:$

\noindent Furthermore, if we assume  $J_{t} = J_{1t} +
\tilde{J}_{2t}$, we have the following result
\begin{theorem}\label{th2}
Under Assumptions \ref{amu}, \ref{ak}, \ref{ah} and we also assume that:\\
(1) $\alpha < 1$ and
$\int_{|x| \leq \varepsilon} x^{2} \nu(dx) = O(\varepsilon^{2 - \alpha}),$
as $\varepsilon \rightarrow 0$;\\
(2) $h = \delta^{\phi}$ with $\frac{h_{n}^{5}}{\delta_{n, T}} =
O_{p}(1),$ $\vartheta(\delta) = \delta^{\eta} , \eta \in (0 , 1)$,
with $\eta / 2 > \phi,~(1 - \alpha\eta) - 1/2 + \phi/2
> 0$ and $\eta(1 - \alpha/2) - 1/2 + \phi/2 > 0,$ then we can obtain
$$\sqrt{\frac{h_{n}}{\delta_{n, T}}} \left(\hat\sigma^2_{n}(x) - \sigma^2(x) -
\frac{1}{2} \frac{(\sigma^{2})^{''}(x)[(K_{1}^{2})^{2} -
K_{1}^{1}K_{1}^{3}]}{K_{1}^{2} - (K_{1}^{1})^{2}}  \cdot h^{2}
\right) \stackrel{\mathcal {S} - \mathcal {L}} \longrightarrow M
\mathscr{N} \left(0 , 2 \frac{\sigma^{6}(x)}{L_{X}(T,x)} \cdot
V_{x}\right).$$
\end{theorem}

\noindent {\textbf{Corollary~2.}} Under Assumptions \ref{amu}, \ref{ak}, \ref{ah} and we also assume that:\\
(1) $\alpha < 1$ and $\int_{|x| \leq \varepsilon} x^{2} \nu(dx) =
O(\varepsilon^{2 - \alpha}),$
as $\varepsilon \rightarrow 0$;\\
(2) $h = \delta^{\phi}$ with $\frac{h_{n}^5
\hat{L}_{X}(T,x)}{\delta_{n, T}} = O_{p}(1),$ $\vartheta(\delta) =
\delta^{\eta} , \eta \in (0 , 1)$, with $\eta / 2 > \phi,~(1 -
\alpha\eta) - 1/2 + \phi/2
> 0$ and $\eta(1 - \alpha/2) - 1/2 + \phi/2 > 0,$ then we can obtain
$$\sqrt{\frac{h_{n} \hat{L}_{X}(T,x) }{\delta_{n,
T}}}\left(\hat\sigma^2_{n}(x) - \sigma^2(x) - \frac{1}{2}
\frac{(\sigma^{2})^{''}(x)[(K_{1}^{2})^{2} -
K_{1}^{1}K_{1}^{3}]}{K_{1}^{2} - (K_{1}^{1})^{2}}  \cdot h^{2}
\right) \Rightarrow \mathscr{N} \left(0 , 2 \sigma^{4}(x) \cdot
V_{x}\right).$$

\begin{remark}
For the local polynomial estimator (\ref{q2}) of order $p$ with
$\hat{\beta}_{0} = \hat\sigma^2_{n}(x)$, under Assumptions
\ref{amu}, \ref{ak}, \ref{ah} and some mild conditions for the
bandwidth $h_{n}$ and the threshold function $\vartheta(\delta)$, we
can obtain
$$\sqrt{\frac{h_{n} \hat{L}_{X}(T,x) }{\delta_{n,
T}}}\left(H (\hat{\beta} - \beta) - \frac{(\sigma^{2})^{(p + 1)}
h_{n^{p + 1}}}{(p + 1)!}S^{- 1}c_{p} \right) \Rightarrow \mathscr{N}
\left(0 , 2 \cdot \sigma^{4}(x) \cdot S^{- 1} S^{*} S^{-
1}\right),$$ where $(\sigma^{2})^{(p + 1)}$ denotes the $(p + 1)$th
derivative of $\sigma^{2}$, $\beta = (\sigma^{2},
(\sigma^{2})^{(1)}, \cdot \cdot \cdot, (\sigma^{2})^{(p)}/p!)^{T},$
$H = diag (1, h_{n}, \cdot \cdot \cdot, h_{n}^{p}),$ $S = (K^{i + j
-2}_{1})$ and $S^{*} = (K^{i + j -2}_{2})$ for $(1 \leq i \leq p +
1; 1 \leq j \leq p + 1),$ $c_{p} = (K^{p + 1}_{1}, \cdot \cdot
\cdot, K^{2p + 1}_{1})^{T}.$
\end{remark}

\begin{remark}
In Mancini and Ren\`{o} \cite{mr}, they only considered the case of
fixed time span $~T = 1$ with $\delta_{n, T} = \frac{1}{n}$ in
Theorem 3.2 and 4.1, and the convergence rate was $\sqrt{n h_{n}}$
for the stable convergence in law and $\sqrt{n h_{n}
\hat{L}_{X}(T,x)}$ for convergence in distribution. Bandi and
Phillips \cite{bp} studied the limiting distribution of the
diffusion estimator in model (\ref{dif}) for the case of time span
$T \longrightarrow \infty$ with $\delta_{n, T} = \frac{T}{n}$ in
Theorem 5, and the convergence rate was $\sqrt{\frac{h_{n}
\hat{L}_{X}(T,x) }{\delta_{n, T}}}$ for convergence in distribution.
In this paper, under two-dimensional asymptotics in both the time
span $T \longrightarrow \infty$ and the sampling interval
$\delta_{n, T} = \frac{T}{n} \longrightarrow 0,$ we derive the local
nonparametric estimator of the diffusion functions for nonstationary
model (\ref{dif}) with convergence rate of $\sqrt{\frac{h_{n}
\hat{L}_{X}(T,x) }{\delta_{n, T}}}$ for convergence in distribution.
We extend the result of Bandi and Phillips \cite{bp} to the
diffusion with jumps model (\ref{om}), especially, the infinite
activity jumps. Meanwhile, we extend the result of Mancini and
Ren\`{o} \cite{mr} in third directions: first, showing the local
polynomial approach to reduce the finite sample bias, which also
extends the result in Moloche \cite{mol} to the diffusion with
jumps, second, considering two-dimensional asymptotics in both the
time span $T \longrightarrow \infty$ and the sampling interval
$\delta_{n, T} = \frac{T}{n} \longrightarrow 0,$ third, posing weak
conditions to the bandwidth parameter $h_{n}$ not allowing for $n
h_{n}^{3} \longrightarrow 0$, which results in the precise bias
representation for the estimator of diffusion function.
\end{remark}

\begin{remark}
If posing weak conditions to the bandwidth parameter $h_{n}$ not
allowing for $n h_{n}^{3} \longrightarrow 0,$  the bias is $h_{n}
\cdot (\sigma^{2})^{'} K_{1}^{1}$ for asymmetric kernels, or
$h_{n}^{2} \cdot ((\sigma^{2})^{'} \frac{s^{'}(x)}{s(x)} +
\frac{1}{2}(\sigma^{2})^{''})K_{1}^{2}$ for symmetric kernels in
Mancini and Ren\`{o} \cite{mr}, where $s(x)$ is the natural scale
funcion, while the bias is $\frac{1}{2}
\frac{(\sigma^{2})^{''}(x)[(K_{1}^{2})^{2} -
K_{1}^{1}K_{1}^{3}]}{K_{1}^{2} - (K_{1}^{1})^{2}}  \cdot h^{2}$ in
this paper with the asymmetric kernel. Hence, the bias in the local
linear case is smaller than the one in the Nadaraya-Watson case in
comparison to the results between this paper and Mancini and
Ren\`{o} \cite{mr} whether the kernel function $K(\cdot)$ is
symmetric or not.
\end{remark}

\begin{remark}
It is very important to consider the choice of the bandwidth $h_{n}$
for the nonparametric estimation. There are many rules of thumb on
selecting the bandwidth, one can refer to Bandi, Corradi and Moloche
\cite{bcm}, Fan and Gijbels \cite{fg}, A\"it-Sahalia and Park
\cite{ap}. Here it would be nice to calculate the optimal bandwidth
based on the mean square error (MSE). The optimal bandwidth of local
threshold nonparametric estimator for model (\cite{om}) based
corollary 1 or 2 is given
$$h_{n, opt} = \left(\frac{4 \delta_{n, T} \left[K_{1}^{2} - (K_{1}^{1})^{2}\right]^{2}}
{\hat{L}_{X}(T,x) \left[(\sigma^{2})^{''}(x)\left((K_{1}^{2})^{2} -
K_{1}^{1}K_{1}^{3}\right)\right]^{2}} \right)^{\frac{1}{5}} =
O_{p}\left(\left(\frac{\delta_{n, T}}{\hat{L}_{X}(T,x)}\right)^{-
\frac{1}{5}}\right).$$ In contrary to Bandi and Nguyen \cite{bn},
they pointed out if $h_{n}^{5} \hat{L}_{X}(T,x) = O_{a.s.}(1)$, then
the features of the nonrandom bias term imply an asymptotic
mean-squared error of order $h_{n}^{4} + \frac{1}{h_{n}
\hat{L}_{X}(T,x)}$ and, in consequence, optimal bandwidth sequences
of order
$$h_{n,opt}={(\hat{L}_{X}(T,x))}^{-1/5}$$ for $\sigma^{2}(x) + \lambda(x) \mathbb{E}_{Y}[c^{2}(x, y)]$
 (P297, equation (13)) in diffusion model
with compound Poisson finite activity jumps. Hence, the optimal
bandwidth in our paper converges to zero faster than that in Bandi
and Nguyen \cite{bn} for diffusion function. To the best of our
knowledge, the optimal bandwidth are not yet derived in the context
of local threshold nonparametric inference for diffusion with jumps,
especially infinite jumps.
\end{remark}

\begin{remark}
Compared with Hanif \cite{han}, this paper considers the local
threshold nonparametric estimation for the diffusion function
$\sigma^{2}(x)$ by disentangling jumps from the observations. It
provides a new method to estimate the components of quadratic
variation separately, especially the volatility contributed by the
Brownian part. With the techniques of lemma 2.1 and 4.3, we deal
with the adaptive and predictable structure of the local
nonparametric threshold estimator conditionally on the
$\sigma-$field generated by diffusion processes, so the lemma 2.2 of
stable convergence in law can be utilized for the estimators. To
some extend, the results for the diffusion with finite and infinite
activity jumps in Theorem 1 and 2 effectively solve the conjecture
for discontinuous variations proposed in the conclusion part of Ye
et al. \cite{ye}, the two-step estimation procedure of the
volatility function in which is a part of (\ref{bias}) in this
paper.
\end{remark}

\section{The proof of main results}

We recall that $\delta = \frac{T}{n}, t_{i} = i\delta , X = Y +
J,~dY_{t}=\mu(X_{t})dt+\sigma(X_{t})dW_{t}.$~~Denote for an integer
$l$,

$\Delta_{i}\hat Y = \Delta_{i}XI_{\{(\Delta_{i}X)^2 \leq
\vartheta(\delta)\}},~\Delta_{i}\hat J =
\Delta_{i}XI_{\{(\Delta_{i}X)^2 > \vartheta(\delta)\}},$

$\Delta_{i}\hat N = I_{\{(\Delta_{i}X)^2 > \vartheta(\delta)\}} ,
\qquad K_{j}^{k}(u) = K^{j}(u)*u^{k}.$

For any bounded process Z we denote by $\bar Z =
\sup_{u\in[0,T]}|Z_{u}|.$ Throughout this article, we use $C$ to
denote a generic constant, which may vary from line to line. By
$\sigma \cdot W$ we denote the stochastic integral of $\sigma$ with
respect to $W$. We denote by $\big(\tau_{j}\big)_{j\in \mathbb{N}}$
the jump instants of $J_{t}$ and by $\tau^{(i)}$ the instant of the
first jump in $(t_{i-1},t_{i}]$, if $\Delta_{i}N \geq 1$.

Before proving our results, we first present some lemmas.

\begin{lemma}\label{l1} (Mancini and Ren\`{o} \cite{mr})
Assume that $J_{t} = \sum_{l=1}^{N_{t}}\gamma_{l}$, where $N_{t}$ is
a doubly stochastic Poisson process with an intensity process
$\lambda(X_{t-})$. If $\lambda (\cdot)$ is bounded, then uniformly
for all $i = 1, \cdots, n$,
$$P((N_{i\delta} - N_{(i-1)\delta}) \geq 1) = O(\delta),$$
$$P((N_{i\delta} - N_{(i-1)\delta}) \geq 2) = O(\delta^2).$$

\end{lemma}

\begin{lemma}\label{l2} (Mancini and Ren\`{o} \cite{mr})
Define
$\prod^{(n)} = \{\frac{i}{n}, i = 1, \cdot \cdot \cdot , n\}$ the
partitions of [0, 1] on which the sums are constructed. There exists
a subsequence $\prod^{(n_{k})} = \{\frac{i}{n_{k}}, i = 1, \cdot
\cdot \cdot , n_{k}\},$ with $\delta_{k} = \frac{1}{n_{k}},$ such
that a.s. for sufficiently small $\delta_{k},$ for all $i = 1, \cdot
\cdot \cdot , n_{k},$ on the set $\{(\Delta_{i}\tilde{J_{2}})^{2}
\leq \vartheta(\delta_{k})\},$ we have
\begin{align*}
& \sum_{s \in (t_{i-1}, t_{i}]}(\Delta_{i}\tilde{J}_{2,s})^{2} \leq 3\vartheta(\delta_{k}),~and~ \\
& \sup_{s \in (t_{i-1}, t_{i}]}\big| \Delta_{i}\tilde{J}_{2,s} \big|
\leq \sqrt{3\vartheta(\delta_{k})}.
\end{align*}

\end{lemma}

\begin{lemma}\label{l3}
Under Assumptions \ref{amu}, \ref{ak}, \ref{ah}, we have
\begin{equation}
\label{3.3 }
\frac{1}{h}\sum_{i=1}^{n}K\big(\frac{X_{t_{i-1}}-x}{h}\big)\big(\frac{X_{t_{i-1}}-x}{h}\big)^{k}\delta~
\stackrel{\mathrm{a.s.}}{\longrightarrow}~\frac{K_{1}^{k}L_{X}(T,x)}{\sigma^2(x)}
\end{equation} where
$K_{1}^{k}L_{X}(T,x)~=~L_{X}(T,x)\int_{R_{+}}K(u)u^kdu$ for all $x$,
as $\delta~ and ~h \to0.$\end{lemma}

\begin{proof}

 For simplicity,  we set $T=1$.  Write
\begin{eqnarray*}
& ~ & \frac{1}{h}\sum_{i=1}^{n}K\big(\frac{X_{t_{i-1}}-x}{h}\big)\big(\frac{X_{t_{i-1}}-x}{h}\big)^{k}\delta\\
& = & \frac{1}{h}\int_{0}^{1}{{K\big(\frac{X_{s-} -
x}{h}\big)}\big(\frac{X_{s-} -
x}{h}\big)^{k}}ds\\
& + & \frac{1}{h}\Big(\sum_{i=1}^{n}
K\big(\frac{X_{t_{i-1}}-x}{h}\big)\big(\frac{X_{t_{i-1}}-x}{h}\big)^{k}\delta
- \int_{0}^{1}{{K\big(\frac{X_{s-} - x}{h}\big)}\big(\frac{X_{s-} -
x}{h}\big)^{k}}ds \Big).
\end{eqnarray*}

By the occupation time formula,
\begin{eqnarray*}
& ~ & \frac{1}{h}\int_{0}^{1}{{K\big(\frac{X_{s-} - x}{h}\big)}\big(\frac{X_{s-} - x}{h}\big)^{k}}ds\\
& = & \frac{1}{h}\int_{0}^{1}{{K\big(\frac{X_{s-} -
x}{h}\big)}\big(\frac{X_{s-} -
x}{h}\big)^{k}}{\frac{d[X]_{s}^{c}}{\sigma^{2}(X_{s-})}}\\
& = & \frac{1}{h}\int_{\mathbb{R}}{K\big(\frac{a-x}{h}\big)\big(\frac{a-x}{h}\big)^{k}\frac{L(a)}{\sigma^{2}(a)}}da\\
& = &
\int_{\mathbb{R}_{+}}{K(u)u^{k}\frac{L(uh+x)}{\sigma^{2}(uh+x)}}du +
\int_{\mathbb{R}_{-}}{K(u)u^{k}\frac{L(uh+x)}{\sigma^{2}(uh+x)}}du,
\end{eqnarray*}
which converges to $\frac{K_{1}^{k}L_{X}(T,x)}{\sigma^2(x)}$ almost surly.

For each $n$,  we define the random sets
$$I_{0,n} = \{i\in \{1 , . . .  , n\}: \Delta_{i}N = 0\},$$ and
$$I_{1,n} = \{i\in \{1 , . . .  , n\}: \Delta_{i}N \neq 0\}.$$

Then
\begin{eqnarray*}
& ~ & \frac{1}{h}\Big(\sum_{i=1}^{n}
K\big(\frac{X_{t_{i-1}}-x}{h}\big)\big(\frac{X_{t_{i-1}}-x}{h}\big)^{k}\delta
- \int_{0}^{1}{{K\big(\frac{X_{s-} - x}{h}\big)}\big(\frac{X_{s-} -
x}{h}\big)^{k}}ds \Big) \\
& = & \frac{1}{h}\sum_{i \in
I_{0,n}}^{}\int_{t_{i-1}}^{t_{i}}{\Big(K\big(\frac{X_{t_{i-1}}-x}{h}\big)\cdot\big(\frac{X_{t_{i-1}}
- x}{h}\big)^{k} - {{K\big(\frac{X_{s-} -
x}{h}\big)}\big(\frac{X_{s-} -
x}{h}\big)^{k}}\Big)}ds \\
 & + & \frac{1}{h}\sum_{i \in
I_{1,n}}^{}\int_{t_{i-1}}^{t_{i}}{\Big(K\big(\frac{X_{t_{i-1}}-x}{h}\big)\cdot\big(\frac{X_{t_{i-1}}
- x}{h}\big)^{k} - {{K\big(\frac{X_{s-} -
x}{h}\big)}\big(\frac{X_{s-} - x}{h}\big)^{k}}\Big)}ds.
\end{eqnarray*}

Noticing that $K(\cdot)$ is bounded supported,
$\frac{1}{h}\sum_{i \in I_{1,n}}\int_{t_{i-1}}^{t_{i}}K\big(\frac{X_{s-}-x}{h}\big)\big(\frac{X_{s-}-x}{h}\big)^{k}ds $ is dominated by $N_{1}\frac{2C\delta}{h}
\stackrel{a.s.}{\longrightarrow} 0$.

$\frac{1}{h}\sum_{i \in
I_{1,n}}^{}\int_{t_{i-1}}^{t_{i}}{\Big(K\big(\frac{X_{t_{i-1}}-x}{h}\big)\cdot\big(\frac{X_{t_{i-1}}
- x}{h}\big)^{k}}ds$ can be written using the mean-value theorem,
and it is a.s. dominated by
\begin{equation}
\label{4.2}
\frac{1}{h}\sum_{i \in
I_{0,n}}\int_{t_{i-1}}^{t_{i}}{\Big|{K_{1}^{k}}^{'}\big(\frac{\tilde{X}_{is}
- x}{h}\big)\Big|\Big|\frac{X_{s} - X_{t_{i-1}}}{h}\Big|}ds\end{equation}
where $\tilde{X}_{is}$ is some point between $X_{i\delta}$
and $X_{s}$ for $i \in I_{0,n}$. Using the property of uniform
boundedness of the increments of X paths when J $\equiv$ 0
(indicated as the UBI property),~(\ref{4.2}) can be a.s. dominated by
\begin{equation}
\label{4.3}
\frac{1}{h}\sum_{i \in
I_{0,n}}\int_{t_{i-1}}^{t_{i}}{\Big|{K_{1}^{k}}^{'}\big(\frac{\tilde{X}_{is}
- x}{h}\big)\Big|\frac{(\delta
ln{\frac{1}{\delta}})^\frac{1}{2}}{h}}ds\end{equation}
 Since $\frac{1}{h}\sum_{i \in
I_{1,n}}\int_{t_{i-1}}^{t_{i}}{\Big|{K_{1}^{k}}^{'}\big(\frac{\tilde{X}_{is}
- x}{h}\big)\Big|\frac{(\delta
ln{\frac{1}{\delta}})^\frac{1}{2}}{h}}ds \leq CN_{1}\frac{(\delta
ln{\frac{1}{\delta}})^\frac{1}{2}}{h}\frac{\delta}{h}
\stackrel{a.s.}{\longrightarrow} 0$, so (\ref{4.3}) has  the same
limit as
\begin{eqnarray*}
& ~ & C\frac{(\delta
ln{\frac{1}{\delta}})^\frac{1}{2}}{h^{2}}\sum_{i =
1}^{n}\int_{t_{i-1}}^{t_{i}}{\Big|{K_{1}^{k}}^{'}\big(\frac{\tilde{X}_{is}
- x}{h}\big)\Big|}ds\\ & = & C\frac{(\delta
ln{\frac{1}{\delta}})^\frac{1}{2}}{h^{2}}\sum_{i =
1}^{n}\int_{t_{i-1}}^{t_{i}}{\Big|{K_{1}^{k}}^{'}\Big(\frac{X_{s-} -
x}{h} + O_{a.s.}(\frac{(\delta
ln{\frac{1}{\delta}})^\frac{1}{2}}{h})\Big)\Big|}ds\\
& = & C\frac{(\delta
ln{\frac{1}{\delta}})^\frac{1}{2}}{h^{2}}\int_{-\infty}^{+\infty}{\Big|{K_{1}^{k}}^{'}\big(\frac{p-x}{h}
+ o_{a.s.}(1)\big)\Big| \frac{L_{X}(T , p)}{\sigma^{2}(p)}} dp\\
& = & C\frac{(\delta
ln{\frac{1}{\delta}})^\frac{1}{2}}{h}\int_{-\infty}^{+\infty}{\Big|{K_{1}^{k}}^{'}\big(q
+ o_{a.s.}(1)\big)\Big| \frac{L_{X}(T , hq + x)}{\sigma^{2}(hq + x)}} dq\\
& \leq & C\frac{(\delta
ln{\frac{1}{\delta}})^\frac{1}{2}}{h}O_{a.s.}\big(\frac{L_{X}(T ,
x)}{\sigma^{2}(x)}\big) = o_{a.s.}(1).
\end{eqnarray*}
The inequality follows from (using the Cauchy-Schwarz
inequality and $K(\cdot)$ bounded support denoted as  $[0 , C]$)
\begin{eqnarray*}
\int_{0}^{C}{\big|{K_{1}^{k}}^{'}(u)\big|}du & \leq &
\int_{0}^{C}{ku^{k-1}K(u)}du + \int_{0}^{C}{u^{k}K^{'}(u)}du\\ &
\leq & \big(\int_{0}^{C}{|u^{k-1}|^{2}}du \cdot
\int_{0}^{C}{(kK(u))^{2}}du\big)^{\frac{1}{2}} +
\big(\int_{0}^{C}{|u^{k}|^{2}}du \cdot
\int_{0}^{C}{(K^{'}(u))^{2}}du\big)^{\frac{1}{2}}\\ & < & \infty.
\end{eqnarray*}

Thus
$$\frac{1}{h}\Big(\sum_{i=1}^{n}
K\big(\frac{X_{t_{i-1}}-x}{h}\big)\big(\frac{X_{t_{i-1}}-x}{h}\big)^{k}\delta
- \int_{0}^{1}{{K\big(\frac{X_{s-} - x}{h}\big)}\big(\frac{X_{s-} -
x}{h}\big)^{k}}ds \Big) \stackrel{a.s.}{\rightarrow} 0.$$
We obtain this lemma.
\end{proof}

\begin{lemma}\label{l4}
Under Assumptions \ref{amu}, \ref{ak}, \ref{ah}  we have
\begin{equation} \label{3.4 }
\sqrt{\frac{n}{h}}\sum_{i=1}^{n}K\big(\frac{X_{t_{i-1}}-x}{h}\big)\big(\frac{X_{t_{i-1}}-x}{h}\big)^{k}
(\Delta_{i}X)^{2}\Big| I_{\{(\Delta_{i}X)^{2} \leq
\vartheta(\delta_{k})\}}  - I_{\{(\Delta_{i}\tilde{J_{2}})^{2} \leq
 4\vartheta(\delta_{k}),~\Delta_{i}N = 0\}} \Big| \stackrel{p} \longrightarrow 0\end{equation}
as $\delta \rightarrow 0$.  \end{lemma}

\begin{proof}
Here,  we consider the case of $J_{t} =J_{1t} + \tilde{J}_{2t}.$

On $\{(\Delta_{i}X)^{2} \leq \vartheta(\delta_{k})\},$ we have
$|\Delta_{i}J| - |\Delta_{i}Y| \leq |\Delta_{i}X| \leq
\sqrt{\vartheta(\delta_{k})},$ and by UBI property of $Y_{t}$, for
small $\delta_{k},~ |\Delta_{i}J| \leq
2\sqrt{\vartheta(\delta_{k})},$
\begin{eqnarray*}
& ~ & \lim_{\delta_{k} \rightarrow 0}
\sqrt{\frac{n}{h}}\sum_{i=1}^{n}K\big(\frac{X_{t_{i-1}}-x}{h}\big)\big(\frac{X_{t_{i-1}}-x}{h}\big)^{k}
(\Delta_{i}X)^{2} I_{\{(\Delta_{i}X)^{2}
 \leq \vartheta(\delta_{k})\}}\\
& \leq & \lim_{\delta_{k} \rightarrow 0}
\sqrt{\frac{n}{h}}\sum_{i=1}^{n}K\big(\frac{X_{t_{i-1}}-x}{h}\big)\big(\frac{X_{t_{i-1}}-x}{h}\big)^{k}
(\Delta_{i}X)^{2} I_{\{(\Delta_{i}J)^{2}
 \leq 4\vartheta(\delta_{k})\}}.
\end{eqnarray*}
However, by the bounded support of $K(\cdot)$ and the UBI
property,
\begin{eqnarray*}
& ~ &
\sqrt{\frac{n}{h}}\sum_{i=1}^{n}K\big(\frac{X_{t_{i-1}}-x}{h}\big)\big(\frac{X_{t_{i-1}}-x}{h}\big)^{k}
(\Delta_{i}X)^{2} I_{\{(\Delta_{i}J)^{2}
 \leq 4\vartheta(\delta_{k}),~\Delta_{i}N \neq 0\}}\\
& \leq & C N_{1} \sqrt{\frac{n}{h}} \cdot \vartheta(\delta_{k})
\rightarrow 0,
\end{eqnarray*}
where $N_{1}$ denotes the number of jumps in [0, 1], thus,
\begin{eqnarray*}
& ~ & \lim_{\delta_{k} \rightarrow 0}
\sqrt{\frac{n}{h}}\sum_{i=1}^{n}K\big(\frac{X_{t_{i-1}}-x}{h}\big)\big(\frac{X_{t_{i-1}}-x}{h}\big)^{k}
(\Delta_{i}X)^{2} I_{\{(\Delta_{i}X)^{2}
 \leq \vartheta(\delta_{k})\}}\\
& \leq & \lim_{\delta_{k} \rightarrow 0}
\sqrt{\frac{n}{h}}\sum_{i=1}^{n}K\big(\frac{X_{t_{i-1}}-x}{h}\big)\big(\frac{X_{t_{i-1}}-x}{h}\big)^{k}
(\Delta_{i}X)^{2} I_{\{(\Delta_{i}J)^{2}
 \leq 4\vartheta(\delta_{k}),~\Delta_{i}N = 0 \}}\\
& = & \lim_{\delta_{k} \rightarrow 0}
\sqrt{\frac{n}{h}}\sum_{i=1}^{n}K\big(\frac{X_{t_{i-1}}-x}{h}\big)\big(\frac{X_{t_{i-1}}-x}{h}\big)^{k}
(\Delta_{i}X)^{2} I_{\{(\Delta_{i}\tilde{J}_{2})^{2}
 \leq 4\vartheta(\delta_{k}),~\Delta_{i}N = 0 \}}.
\end{eqnarray*}

\noindent It is sufficient to prove
$$\lim_{\delta_{k} \rightarrow 0}\sqrt{\frac{n}{h}}\sum_{i=1}^{n}K\big(\frac{X_{t_{i-1}}-x}{h}\big)\big(\frac{X_{t_{i-1}}-x}{h}\big)^{k}
(\Delta_{i}X)^{2} \Big(I_{\{(\Delta_{i}\tilde{J})^{2}
 \leq 4\vartheta(\delta_{k}),~\Delta_{i}N = 0 \}} - I_{\{(\Delta_{i}X)^{2}
 \leq \vartheta(\delta_{k})\}} \Big) = 0$$
which can be similarly proved using lemma 4.2 as the technical
details for Lemma 3 in Mancini and Ren\`{o} (\cite{mr}) with $K(
\cdot )( \cdot )^{k}$ instead of $K( \cdot )$.
\end{proof}

\subsection{The proof of Theorem \ref{th1}}
Set $T = 1$ and $\sigma(X_{s}) =:
\sigma_{s}, \sigma(X_{(i-1)\delta}) =: \sigma_{i-1}$. Theorem
1 in Mancini (\cite{m1}) means that  $I_{\{(\Delta_{i}X)^2 \leq
\vartheta(\delta)\}}(\omega) = I_{\{\Delta_{i}N=0\}}(\omega)$), then
\begin{eqnarray*}
& &\sqrt{nh}(\hat{\sigma}_{n}^{2}(x) - \sigma^{2}(x)) \\& = &
\frac{\sqrt{\frac{n}{h}}\sum_{i=1}^{n}K(\frac{X_{t_{i-1}}-x}{h})\{\frac{\delta
S_{n,2}}{h^2}-(\frac{X_{t_{i-1}}-x}{h})\frac{\delta
S_{n,1}}{h}\}\big((\Delta_{i}Y)^{2} -
\sigma^{2}(x)\delta\big)}{\frac{1}{h}\sum_{i=1}^{n}K(\frac{X_{t_{i-1}}-x}{h})\{\frac{\delta
S_{n,2}}{h^2}-(\frac{X_{t_{i-1}}-x}{h})\frac{\delta
S_{n,1}}{h}\}\delta}\\
& ~ & -
\frac{\sqrt{\frac{n}{h}}\sum_{i=1}^{n}K(\frac{X_{t_{i-1}}-x}{h})\{\frac{\delta
S_{n,2}}{h^2}-(\frac{X_{t_{i-1}}-x}{h})\frac{\delta
S_{n,1}}{h}\}(\Delta_{i}Y)^{2}I_{\{\Delta_{i}N \neq
0\}}}{\frac{1}{h}\sum_{i=1}^{n}K(\frac{X_{t_{i-1}}-x}{h})\{\frac{\delta
S_{n,2}}{h^2}-(\frac{X_{t_{i-1}}-x}{h})\frac{\delta
S_{n,1}}{h}\}\delta}
\end{eqnarray*}
Similar to the proof of Lemma \ref{l3},  last term is
$$O_{a.s.}\Big(N_{1}^{2}\sqrt{\frac{\delta}{h}}ln(\frac{1}{\delta})\frac{\delta}{h}\frac{\sigma^{4}(x)}{L_{X}^{2}(T,x)[K_{1}^{2}
- (K_{1}^{1})^{2}]}\Big) \stackrel{a.s.}{\longrightarrow} 0.$$

\noindent By Jacod's stable convergence theorem with the help of
lemmas 2.1 and 4.3, we first show that the numerator of
\begin{equation}
\label{4.4}
\frac{\sqrt{\frac{n}{h}}\sum_{i=1}^{n}K(\frac{X_{t_{i-1}}-x}{h})\{\frac{\delta
S_{n,2}}{h^2}-(\frac{X_{t_{i-1}}-x}{h})\frac{\delta
S_{n,1}}{h}\}\big((\Delta_{i}Y)^{2} -
\sigma^{2}(x)\delta\big)}{\frac{1}{h}\sum_{i=1}^{n}K(\frac{X_{t_{i-1}}-x}{h})\{\frac{\delta
S_{n,2}}{h^2}-(\frac{X_{t_{i-1}}-x}{h})\frac{\delta
S_{n,1}}{h}\}\delta}\end{equation} converges stably in law to
$M_{1}$ with the asymptotic bias
$\frac{1}{2}(\sigma^{2})^{''}(x)[(K_{1}^{2})^{2} -
K_{1}^{1}K_{1}^{3}] \cdot \Big(\frac{L_{X}(T,
x)}{\sigma^{2}(x)}\Big)^{2} \cdot h^{2}$, where $M_{1}$ is a
Gaussian martingale defined on an extension $\big(\tilde{\Omega} ,
\tilde{P} , \tilde{\mathscr{F}}\big)$ of our filtered probability
space and having $\tilde{E}[M_{1}^{2}|\mathscr{F}] =
\frac{2}{\sigma^{2}(x)}\cdot L_{X}^{3}(T,x)\cdot V_{x}^{'}$ with
$V_{x}^{'} = K_{2}^{0}\cdot\big(K_{1}^{2}\big)^{2} +
K_{2}^{2}\cdot\big(K_{1}^{1}\big)^{2} - 2K_{2}^{1}\cdot
K_{1}^{2}\cdot K_{1}^{1}.$

\noindent Using the It\^{o} formula on $(\Delta_{i}Y)^{2}$, we have
\begin{eqnarray*}
& ~ &
\sqrt{\frac{n}{h}}\sum_{i=1}^{n}K(\frac{X_{t_{i-1}}-x}{h})\{\frac{\delta
S_{n,2}}{h^2}-(\frac{X_{t_{i-1}}-x}{h})\frac{\delta
S_{n,1}}{h}\}\big((\Delta_{i}Y)^{2} - \sigma^{2}(x)\delta\big)\\ & =
&
\sqrt{\frac{n}{h}}\sum_{i=1}^{n}K(\frac{X_{t_{i-1}}-x}{h})\{\frac{\delta
S_{n,2}}{h^2}-(\frac{X_{t_{i-1}}-x}{h})\frac{\delta
S_{n,1}}{h}\}\Big[2\int_{t_{i-1}}^{t_{i}}{(Y_{s} -
Y_{(i-1)\delta})\mu_{s}}ds \\
& & {} + 2\int_{t_{i-1}}^{t_{i}}{(Y_{s} -
Y_{(i-1)\delta})\sigma_{s}}dW_{s}\Big] +
\sqrt{\frac{n}{h}}\sum_{i=1}^{n}K(\frac{X_{t_{i-1}}-x}{h})\{\frac{\delta
S_{n,2}}{h^2}\\ & &-(\frac{X_{t_{i-1}}-x}{h})\frac{\delta
S_{n,1}}{h}\}\int_{t_{i-1}}^{t_{i}}{(\sigma^{2}_{s} -
\sigma^{2}(x))}ds\\
& =: & \sum_{i=1}^{n}q_{i} +
\sqrt{\frac{n}{h}}\sum_{i=1}^{n}K(\frac{X_{t_{i-1}}-x}{h})\{\frac{\delta
S_{n,2}}{h^2}-(\frac{X_{t_{i-1}}-x}{h})\frac{\delta
S_{n,1}}{h}\}\int_{t_{i-1}}^{t_{i}}{(\sigma^{2}_{s} -
\sigma^{2}(x))}ds.
\end{eqnarray*}

\noindent {\textbf{First Step:} the stable convergence in law for
the numerator of the estimator.

\noindent For the term $\sum_{i=1}^{n}q_{i}.$ Divide $q_{i}$ by
$L_{X}(T, x)$ as $q_{i}^{'}.$ For simplicity in the detailed proof,
denote
\begin{eqnarray*}
K^{\divideontimes}_{i-1} & = &
\frac{K(\frac{X_{t_{i-1}}-x}{h})\{\frac{\delta
S_{n,2}}{h^2}-(\frac{X_{t_{i-1}}-x}{h})\frac{\delta
S_{n,1}}{h}\}}{L_{X}(T, x)}\\
& \stackrel{\mathrm{a.s.}}{\longrightarrow} &
\frac{K(\frac{X_{t_{i-1}}-x}{h}) \cdot \frac{K_{1}^{2}L_{X}(T,
x)}{\sigma^{2}(x)} -
K(\frac{X_{t_{i-1}}-x}{h})(\frac{X_{t_{i-1}}-x}{h}) \cdot
\frac{K_{1}^{1}L_{X}(T, x)}{\sigma^{2}(x)}}{L_{X}(T, x)}\\
& = & K(\frac{X_{t_{i-1}}-x}{h}) \cdot
\frac{K_{1}^{2}}{\sigma^{2}(x)} -
K(\frac{X_{t_{i-1}}-x}{h})(\frac{X_{t_{i-1}}-x}{h}) \cdot
\frac{K_{1}^{1}}{\sigma^{2}(x)} := K^{\star}_{i-1}.
\end{eqnarray*}
 So there exists an integer $m$ such that
$K^{\divideontimes}_{i-1} = K^{\star}_{i-1} + o_{\mathrm{a.s.}}(1)$
when $n>m$. In the following proof, we will substitute
$K(\frac{X_{t_{i-1}}-x}{h}) \cdot \frac{K_{1}^{2}}{\sigma^{2}(x)} -
K(\frac{X_{t_{i-1}}-x}{h})(\frac{X_{t_{i-1}}-x}{h}) \cdot
\frac{K_{1}^{1}}{\sigma^{2}(x)}$ for $K^{\star}_{i-1}$ for the
sample sizes $n>m$ and assume the sample sizes $n>m$.

In fact, \begin{eqnarray*} q_{i}^{'} & = &
\frac{\sqrt{\frac{n}{h}}K(\frac{X_{t_{i-1}}-x}{h})\{\frac{\delta
S_{n,2}}{h^2}-(\frac{X_{t_{i-1}}-x}{h})\frac{\delta
S_{n,1}}{h}\}}{L_{X}(T, x)} \times\\
& ~ & \times \Big[2\int_{t_{i-1}}^{t_{i}}{(Y_{s} -
Y_{(i-1)\delta})\mu_{s}}ds + 2\int_{t_{i-1}}^{t_{i}}{(Y_{s} -
Y_{(i-1)\delta})\sigma_{s}}dW_{s}\Big]\\
& \stackrel{\mathrm{a.s.}}{\longrightarrow} &
\sqrt{\frac{n}{h}}K(\frac{X_{t_{i-1}}-x}{h})\Big\{
\frac{K_{1}^{2}}{\sigma^{2}(x)} - \frac{X_{t_{i-1}}-x}{h} \cdot
\frac{K_{1}^{1}}{\sigma^{2}(x)}\Big\}\Big[2\int_{t_{i-1}}^{t_{i}}{(Y_{s}
- Y_{(i-1)\delta})\mu_{s}}ds \\ & &+ 2\int_{t_{i-1}}^{t_{i}}{(Y_{s}
-
Y_{(i-1)\delta})\sigma_{s}}dW_{s}\Big]\\
& = &
\sqrt{\frac{n}{h}}K^{\star}_{i-1}\Big[2\int_{t_{i-1}}^{t_{i}}{(Y_{s}
- Y_{(i-1)\delta})\mu_{s}}ds + 2\int_{t_{i-1}}^{t_{i}}{(Y_{s} -
Y_{(i-1)\delta})\sigma_{s}}dW_{s}\Big],
\end{eqnarray*}
where $K^{\star}_{i-1}$ is measurable with respect to the
$\sigma$-algebra generated by $\{X_u ,0 \leq u \leq t_{i-1} \}.$

Jacod's stable convergence theorem tell us  that the following arguments,
\begin{align*}
& S_{1} = \sum_{i=1}^{n}E_{i-1}[q_{i}^{'}]
\stackrel{P}{\rightarrow}0,\\
& S_{2} = \sum_{i=1}^{n}\big(E_{i-1}[{q_{i}^{'}}^{2}] -
E_{i-1}^{2}[q_{i}^{'}]\big)\stackrel{P}{\rightarrow}\frac{2}{\sigma^{2}(x)}\cdot
L_{X}(T,x)\big[K_{2}^{0}\cdot\big(K_{1}^{2}\big)^{2} +
K_{2}^{2}\cdot\big(K_{1}^{1}\big)^{2} - 2K_{2}^{1}\cdot
K_{1}^{2}\cdot K_{1}^{1}\big],\\
& S_{3} = \sum_{i=1}^{n}E_{i-1}[{q_{i}^{'}}^{4}]
\stackrel{P}{\rightarrow}0,\\
& S_{4} = \sum_{i=1}^{n}E_{i-1}[q_{i}^{'}\Delta_{i}H]
\stackrel{P}{\rightarrow}0,
\end{align*}
implies
 $\sum_{i=1}^{n}q_{i}^{'}
\stackrel{st}{\rightarrow}M_{1}$, where either $H = W$ or $H$ is any
bounded martingale orthogonal (in the martingale sense) to $W$ and
$E_{i-1}[~\cdot~] = E[~\cdot~| X_{t_{i-1}}].$ Remark that $\mu$ is
assumed to be c\`{a}dl\`{a}g, therefore we know that it is locally
bounded on $[0 , T]$. By localizing, we can assume that $\mu$ is
a.s. bounded on $[0 , T].$

\noindent {\textbf{For $S_{1}$},
\begin{eqnarray*}
\big|S_{1}\big| & = &
\Big|\sqrt{\frac{n}{h}}\sum_{i=1}^{n}E_{i-1}\Big\{K^{\star}_{i-1}\Big[2\int_{t_{i-1}}^{t_{i}}{(Y_{s}
- Y_{(i-1)\delta})\mu_{s}}ds + 2\int_{t_{i-1}}^{t_{i}}{(Y_{s} -
Y_{(i-1)\delta})\sigma_{s}}dW_{s}\Big]\Big\}\Big|\\
& = &
\Big|\sqrt{\frac{n}{h}}\sum_{i=1}^{n}K^{\star}_{i-1}E_{i-1}\Big[2\int_{t_{i-1}}^{t_{i}}{(Y_{s}
- Y_{(i-1)\delta})\mu_{s}}ds + 2\int_{t_{i-1}}^{t_{i}}{(Y_{s} -
Y_{(i-1)\delta})\sigma_{s}}dW_{s}\Big]\Big|\\
& = &
\Big|\sqrt{\frac{n}{h}}\sum_{i=1}^{n}K^{\star}_{i-1}E_{i-1}\Big[2\int_{t_{i-1}}^{t_{i}}{(Y_{s}
- Y_{(i-1)\delta})\mu_{s}}ds\Big]\Big|\\
& \leq &
2\sqrt{\frac{n}{h}}\sum_{i=1}^{n}\big|K^{\star}_{i-1}\big|E_{i-1}\Big[\int_{t_{i-1}}^{t_{i}}{\big|(Y_{s}
- Y_{(i-1)\delta})\mu_{s}}\big|ds\Big]\\
& \leq & 2
\sqrt{\frac{n}{h}}\sum_{i=1}^{n}\big|K^{\star}_{i-1}\big|\cdot
\max_{1\leq i \leq n}E_{i-1}\Big[C\Big(\delta
ln\big(\frac{1}{\delta}\big)\Big)^{\frac{1}{2}}
\int_{t_{i-1}}^{t_{i}}{|\mu_{s}|ds}\Big]\\
& \leq &
C\cdot\Big[\frac{1}{h}\sum_{i=1}^{n}K(\frac{X_{t_{i-1}}-x}{h})\delta
\cdot \frac{K_{1}^{2}}{\sigma^{2}(x)} \\ & &+ sign(K_{1}^{1}) \cdot
\frac{1}{h}\sum_{i=1}^{n}K(\frac{X_{t_{i-1}}-x}{h})(\frac{X_{t_{i-1}}-x}{h})\delta
\cdot
\frac{K_{1}^{1}}{\sigma^{2}(x)}\Big]\cdot\big(\sqrt{h}(ln\frac{1}{\delta})^\frac{1}{2}\big)\\
& = & C \cdot \Big[\frac{L_{X}(T, x)}{\sigma^{2}(x)} \cdot
\frac{K_{1}^{2}}{\sigma^{2}(x)} + sign(K_{1}^{1}) \cdot
\frac{L_{X}(T, x) K_{1}^{2}}{\sigma^{2}(x)} \cdot
\frac{K_{1}^{1}}{\sigma^{2}(x)}\Big] \cdot \big(\sqrt{h}(ln\frac{1}{\delta})^\frac{1}{2}\big)\\
 & = & C\big(\sqrt{h}(ln\frac{1}{\delta})^\frac{1}{2}\big)
\rightarrow 0.
\end{eqnarray*}
by the measurability $K^{\star}_{i-1}$ with respect to $\sigma\{X_u
,0 \leq u \leq t_{i-1} \}$ in the second equation, the martingale
property of stochastic integral in the third equation, the UBI
property in the second inequation, the expression of
$K^{\star}_{i-1}$, the assumption 3.3 and the boundness of $\mu$ in
the third inequation.

\noindent {\textbf{For $S_{2}$},
\begin{eqnarray*}
S_{2} & = & \sum_{i=1}^{n}\big(E_{i-1}[{q_{i}^{'}}^{2}] -
E_{i-1}^{2}[q_{i}^{'}]\big)\\
& = &
4\frac{n}{h}\sum_{i=1}^{n}{K_{i-1}^{\star}}^{2}\big[E_{i-1}\big[\big(\int_{t_{i-1}}^{t_{i}}{(Y_{s}
- Y_{(i-1)\delta})}\mu_{s}ds\big)^{2}\big]\\
& + & 2E_{i-1}\big[\int_{t_{i-1}}^{t_{i}}{(Y_{s} -
Y_{(i-1)\delta})}\mu_{s}ds\int_{t_{i-1}}^{t_{i}}{(Y_{s} -
Y_{(i-1)\delta})}\sigma_{s}dW_{s}\big]\\
& + & E_{i-1}\big[\int_{t_{i-1}}^{t_{i}}{(Y_{s} -
Y_{(i-1)\delta})^{2}}\sigma_{s}^{2}ds\big] -
E^{2}_{i-1}\big[\int_{t_{i-1}}^{t_{i}}{(Y_{s} -
Y_{(i-1)\delta})}\mu_{s}ds\ \big]\big]\\
& =: & S_{2,1} + S_{2,2} + S_{2,3} + S_{2,4}.
\end{eqnarray*}
By H\"{o}lder and
Burkholder-Davis-Gundy  inequality,   $S_{2,3}$  is larger than the others (which has the lowest
infinitesimal order). Here we only deal with the dominant one,
others are neglected. For $S_{2,3}$, it consists of three terms by
an expansion of ${(Y_{s} - Y_{(i-1)\delta})^{2}}$, of which we only
need to consider the lowest infinitesimal order one. Due to
H\"{o}lder and Burkholder-Davis-Gundy inequality again, it is
sufficient to prove the convergence in probability of
\begin{eqnarray*}
& ~ &
4\frac{n}{h}\sum_{i=1}^{n}{K_{i-1}^{\star}}^{2}\int_{t_{i-1}}^{t_{i}}
{E_{i-1}\big[\big(\int_{t_{i-1}}^{s}\sigma_{u}dW_{u}\big)^{2}\sigma_{s}^{2}\big]ds}\\
& \stackrel{a.s.}{\longrightarrow} & \frac{2}{\sigma^{2}(x)}\cdot
L_{X}(T,x)\big[K_{2}^{0}\cdot\big(K_{1}^{2}\big)^{2} +
K_{2}^{2}\cdot\big(K_{1}^{1}\big)^{2} - 2K_{2}^{1}\cdot
K_{1}^{2}\cdot K_{1}^{1}\big].
\end{eqnarray*}

\noindent To show it, we can prove the following five arguments:

\noindent
{\textbf{(D1)}~$\frac{n}{h}\sum_{i=1}^{n}{K_{i-1}^{\star}}^{2}\int_{t_{i-1}}^{t_{i}}
{E_{i-1}\big[\big(\int_{(i-1)\delta}^{s}\sigma_{u}dW_{u}\big)^{2}(\sigma_{s}^{2}
- \sigma_{i-1}^{2})}\big]ds \stackrel{a.s.}{\longrightarrow} 0;$
\medskip

\noindent
{\textbf{(D2)}~$\frac{1}{h}\sum_{i=1}^{n}{K_{i-1}^{\star}}^{2}\int_{t_{i-1}}^{t_{i}}
{\big(nE_{i-1}\big[\int_{(i-1)\delta}^{s}\sigma_{u}^{2}du\big]\sigma_{i-1}^{2}
- \frac{\sigma_{i-1}^{4}}{2}}\big)ds
\stackrel{a.s.}{\longrightarrow} 0;$
\medskip

\noindent
{\textbf{(D3)}~$\frac{1}{h}\sum_{i=1}^{n}{K_{i-1}^{\star}}^{2}\int_{t_{i-1}}^{t_{i}}
(\sigma_{i-1}^{4} - \sigma_{s}^{4}) \stackrel{a.s.}{\longrightarrow}
0;$
\medskip

\noindent
{\textbf{(D4)}~$\frac{1}{h}\sum_{i=1}^{n}\int_{t_{i-1}}^{t_{i}}
\big({K_{i-1}^{\star}}^{2}\frac{\sigma_{s}^{4}}{2} -
{K_{s}^{\star}}^{2}\frac{\sigma_{s}^{4}}{2}\big)ds
\stackrel{a.s.}{\longrightarrow} 0;$
\medskip

\noindent
{\textbf{(D5)}~$\frac{4}{h}\sum_{i=1}^{n}\int_{t_{i-1}}^{t_{i}}
\big({K_{s}^{\star}}^{2}\frac{\sigma_{s}^{4}}{2}\big)ds
\stackrel{a.s.}{\longrightarrow} \frac{2}{\sigma^{2}(x)}\cdot
L_{X}(T,x)\big[K_{2}^{0}\cdot\big(K_{1}^{2}\big)^{2} +
K_{2}^{2}\cdot\big(K_{1}^{1}\big)^{2} - 2K_{2}^{1}\cdot
K_{1}^{2}\cdot K_{1}^{1}\big].$
\medskip

\noindent $\bf{For~~D1:}$

\noindent Applying the mean-value theorem for $\sigma^{2}$,
neglecting the terms with $i \in I_{1,n}$ similarly as that in
Proposition 3.1 and bounding $|X_{s} - X_{(i-1)\delta}|$ by the UBI
property when $i \in I_{0,n}$, for the sum in (D1) we can reach
\begin{eqnarray*}
& ~ &
\frac{n}{h}\sum_{i=1}^{n}{K_{i-1}^{\star}}^{2}\int_{t_{i-1}}^{t_{i}}
{E_{i-1}\big[\big(\int_{(i-1)\delta}^{s}\sigma_{u}dW_{u}\big)^{2}}\big]ds
* \sup_{x}|(\sigma^{2})^{'}(x)| * \sqrt{\delta
ln\frac{1}{\delta}}\\
& = &
\frac{n}{h}\sum_{i=1}^{n}{K_{i-1}^{\star}}^{2}\int_{t_{i-1}}^{t_{i}}
{E_{i-1}\big[\int_{(i-1)\delta}^{s}\sigma_{u}^{2}d{u}}\big]ds
* \sup_{x}|(\sigma^{2})^{'}(x)| * \sqrt{\delta
ln\frac{1}{\delta}}\\
& = &
\frac{n}{h}\sum_{i=1}^{n}{K_{i-1}^{\star}}^{2}\int_{t_{i-1}}^{t_{i}}\big({t
- (i-1)\delta}\big)ds
* \sup_{x}|(\sigma^{2})^{'}(x)|^{2} * \sqrt{\delta
ln\frac{1}{\delta}}\\
 & = & O_{a.s.}\big(\frac{1}{h}\sum_{i
\in I_{0,n}}{K_{i-1}^{\star}}^{2}\delta\cdot\sqrt{\delta
ln\frac{1}{\delta}}\big)\\ & = &
O_{a.s.}\big(\frac{1}{h}\sum_{i=1}^{n}{K_{i-1}^{\star}}^{2}\delta
\cdot \sqrt{\delta
ln\frac{1}{\delta}}\big)\\
& = & O_{a.s.}\Big(\frac{K_{2}^{0}\cdot\big(K_{1}^{2}\big)^{2} +
K_{2}^{2}\cdot\big(K_{1}^{1}\big)^{2} - 2K_{2}^{1}\cdot
K_{1}^{2}\cdot K_{1}^{1}}{\sigma^{6}(x)}L_{X}(T, x)\Big) \cdot
\sqrt{\delta ln\frac{1}{\delta}} \stackrel{a.s.}{\longrightarrow} 0.
\end{eqnarray*}

\noindent $\bf{For~~D2:}$

\noindent By the Taylor expansion,
$$\int_{(i-1)\delta}^{s}\sigma_{u}^{2}du = \sigma_{i-1}^{2}(s -
t_{i-1}) + (\sigma^{2})^{'}(\tilde{X}_{is})\frac{(s -
t_{i-1})^{2}}{2}.$$

\noindent For the sum in (D2) we can obtain
\begin{eqnarray*}
& ~ &
\frac{1}{h}\sum_{i=1}^{n}{K_{i-1}^{\star}}^{2}\int_{t_{i-1}}^{t_{i}}
{\sigma_{i-1}^{2}\big[nE_{i-1}\big[\int_{(i-1)\delta}^{s}\sigma_{u}^{2}du\big]
- \frac{\sigma_{i-1}^{2}}{2}}\big]ds\\
& = &
\frac{1}{h}\sum_{i=1}^{n}{K_{i-1}^{\star}}^{2}\int_{t_{i-1}}^{t_{i}}
{\sigma_{i-1}^{2}\big(n\sigma_{i-1}^{2}(s
- t_{i-1}) +
nE_{i-1}[(\sigma^{2})^{'}(\tilde{X}_{is})]\times\frac{(s -
t_{i-1})^{2}}{2} - \frac{\sigma_{i-1}^{2}}{2}\big)}ds\\
& = &
\frac{1}{h}\sum_{i=1}^{n}{K_{i-1}^{\star}}^{2}\int_{t_{i-1}}^{t_{i}}
nE_{i-1}[(\sigma^{2})^{'}(\tilde{X}_{is})]\times\frac{(s
- t_{i-1})^{2}}{2}ds\\
& = &
\frac{1}{h}\sum_{i=1}^{n}{K_{i-1}^{\star}}^{2}\int_{t_{i-1}}^{t_{i}}n\times\frac{(s
- t_{i-1})^{2}}{2}ds * \sup_{x}|(\sigma^{2})^{'}(x)| \\
& = &
O_{a.s.}\Big(\delta\frac{1}{h}\sum_{i=1}^{n}{K_{i-1}^{\star}}^{2}\delta\Big)\\
& = & O_{a.s.}\Big(\frac{K_{2}^{0}\cdot\big(K_{1}^{2}\big)^{2} +
K_{2}^{2}\cdot\big(K_{1}^{1}\big)^{2} - 2K_{2}^{1}\cdot
K_{1}^{2}\cdot K_{1}^{1}}{\sigma^{6}(x)}L_{X}(T, x)\Big) \cdot
\delta \longrightarrow 0.
\end{eqnarray*}

\noindent $\bf{For~~D3:}$

 Neglecting the terms with $i \in I_{1,n}$ proceeding as
Lemma \ref{l3}, (D3) is a.s. dominated by
\begin{eqnarray*}
& ~ & \overline{(\sigma_{s}^{4})^{'}}\sqrt{\delta
ln\frac{1}{\delta}}\frac{1}{h}\sum_{i=1}^{n}{K_{i-1}^{\star}}^{2}\delta\\
& = & \overline{(\sigma_{s}^{4})^{'}}\sqrt{\delta
ln\frac{1}{\delta}} \cdot
\Big(\frac{K_{2}^{0}\cdot\big(K_{1}^{2}\big)^{2} +
K_{2}^{2}\cdot\big(K_{1}^{1}\big)^{2} - 2K_{2}^{1}\cdot
K_{1}^{2}\cdot K_{1}^{1}}{\sigma^{6}(x)}L_{X}(T, x)\Big)\\
& \stackrel{a.s.}{\longrightarrow} & 0,
\end{eqnarray*}

\noindent $\bf{For~~D4:}$

\noindent Since  $\sigma_{s}^{4}$ is bounded  almost surely, similar
to the proof of  Lemma \ref{l3} and $\bf{D1}$,  $\bf{D4}$ is obtained.

\noindent $\bf{For~~D5:}$

\noindent Using the occupation time formula, we obtain
\begin{eqnarray*}
& &\frac{2}{h}\sum_{i=1}^{n}\int_{t_{i-1}}^{t_{i}}\big({K_{s}^{\star}}^{2}
\cdot \sigma_{s}^{4}\big)ds\\ & = &
\frac{2}{h}\sum_{i=1}^{n}\int_{t_{i-1}}^{t_{i}}\Big(K(\frac{X_{s}-x}{h})
\cdot \frac{K_{1}^{2}}{\sigma^{2}(x)} -
K(\frac{X_{s}-x}{h})(\frac{X_{s}-x}{h}) \cdot
\frac{K_{1}^{1}}{\sigma^{2}(x)}\Big)^{2} \cdot \sigma^{4}(X_{s})ds\\
& = & \frac{2}{h}\int_{0}^{1}K^{2}(\frac{X_{s}-x}{h})
\cdot \sigma^{4}(X_{s})ds \cdot \frac{(K_{1}^{2})^{2}}{\sigma^{4}(x)}\\
& ~ & - \frac{4}{h}\int_{0}^{1}K^{2}(\frac{X_{s} -
x}{h})(\frac{X_{s} - x}{h}) \cdot \sigma^{4}(X_{s})ds \cdot
\frac{K_{1}^{2}K_{1}^{1}}{\sigma^{4}(x)}\\
& ~ & + \frac{2}{h}\int_{0}^{1}K^{2}(\frac{X_{s} -
x}{h})(\frac{X_{s} - x}{h})^{2} \cdot \sigma^{4}(X_{s})ds \cdot
\frac{(K_{1}^{1})^{2}}{\sigma^{4}(x)}\\
& = & \frac{2}{h}\int_{\mathbb{R}}K^{2}(\frac{a-x}{h})
\cdot \sigma^{4}(a) \frac{L(a)}{\sigma^{2}(a)}ds \cdot \frac{(K_{1}^{2})^{2}}{\sigma^{4}(x)}\\
& ~ & - \frac{4}{h}\int_{\mathbb{R}}K^{2}(\frac{a - x}{h})(\frac{a -
x}{h}) \cdot \sigma^{4}(a) \frac{L(a)}{\sigma^{2}(a)}ds \cdot
\frac{K_{1}^{2}K_{1}^{1}}{\sigma^{4}(x)}\\
& ~ & + \frac{2}{h}\int_{\mathbb{R}}K^{2}(\frac{a - x}{h})(\frac{a -
x}{h})^{2} \cdot \sigma^{4}(a) \frac{L(a)}{\sigma^{2}(a)}ds \cdot
\frac{(K_{1}^{1})^{2}}{\sigma^{4}(x)}\\
& \stackrel{a.s.}{\longrightarrow} & \frac{2}{\sigma^{2}(x)}\cdot
L_{X}(T,x)\big[K_{2}^{0}\cdot\big(K_{1}^{2}\big)^{2} +
K_{2}^{2}\cdot\big(K_{1}^{1}\big)^{2} - 2K_{2}^{1}\cdot
K_{1}^{2}\cdot K_{1}^{1}\big].
\end{eqnarray*}

\noindent {\textbf{For $S_{3}$}, let us come back to the proof of
{\textbf{$S_{3}$}. Using BDG and H\"{o}lder inequalities, we have
\begin{eqnarray*}
& ~ &
\frac{n^{2}}{h^{2}}\sum_{i=1}^{n}{K_{i-1}^{\star}}^{4}E_{i-1}\big[\big(\int_{t_{i-1}}^{t_{i}}{(Y_{s}
- Y_{(i-1)\delta})}\sigma_{s}dW_{s}\big)^{4}\big]\\
& \leq &
\frac{n^{2}}{h^{2}}\sum_{i=1}^{n}{K_{i-1}^{\star}}^{4}E_{i-1}\big[\sup_{s
\in [t_{i-1}, t_{i}]}\big(\int_{t_{i-1}}^{s}{(Y_{s} -
Y_{(i-1)\delta})}\sigma_{s}dW_{s}\big)^{4}\big]\\
& \leq & C \cdot
\frac{n^{2}}{h^{2}}\sum_{i=1}^{n}{K_{i-1}^{\star}}^{4}E_{i-1}\big[\big(\int_{t_{i-1}}^{t_{i}}{(Y_{s}
- Y_{(i-1)\delta})^{2}}\sigma_{s}^{2}d{s}\big)^{2}\big]\\
& \leq & C \cdot
\frac{n^{2}}{h^{2}}\sum_{i=1}^{n}{K_{i-1}^{\star}}^{4}E_{i-1}\big[\int_{t_{i-1}}^{t_{i}}{(Y_{s}
- Y_{(i-1)\delta})^{4}}\sigma_{s}^{4}d{s}\big] \cdot
E_{i-1}\big[\int_{t_{i-1}}^{t_{i}}1^{2}ds\big]\\
& = & C \cdot
\frac{n}{h^{2}}\sum_{i=1}^{n}{K_{i-1}^{\star}}^{4}E_{i-1}\big[\int_{t_{i-1}}^{t_{i}}{(Y_{s}
- Y_{(i-1)\delta})^{4}}\sigma_{s}^{4}ds\big]\\
& \leq & C \cdot
\frac{n}{h^{2}}\sum_{i=1}^{n}{K_{i-1}^{\star}}^{4}E_{i-1}\big[\int_{t_{i-1}}^{t_{i}}
\big(\int^{s}_{(i-1)\delta}\sigma_{u}dW_{u}\big)^{4}ds\big]\\
& \leq & C \cdot
\frac{n}{h^{2}}\sum_{i=1}^{n}{K_{i-1}^{\star}}^{4}\Big[\int_{t_{i-1}}^{t_{i}}
E_{i-1}\big(\int^{s}_{(i-1)\delta}\sigma^{4}_{u}d{u}\big)ds
\cdot E_{i-1}\big[\int_{t_{i-1}}^{s}1^{2}ds\big]\Big]\\
& \leq & \frac{1}{h^{2}}\sum_{i=1}^{n}{K_{i-1}^{\star}}^{4} \cdot
\delta^{2}\\
& = & \frac{\delta}{h} \cdot \frac{1}{h}
\sum_{i=1}^{n}{K_{i-1}^{\star}}^{4} \cdot \delta =
O_{a.s.}\big(\frac{\delta}{h}\big) \longrightarrow 0.
\end{eqnarray*}

\noindent {\textbf{For $S_{4}$},

Set  $\Delta_{i}Z := \int_{t_{i-1}}^{t_{i}}(Y_{s} -
Y_{i-1})dY_{s}.$ If $H = W,$ then
\begin{eqnarray*}
\sum_{i=1}^{n}E_{i-1}[q_{i}\Delta_{i}H] & = &
2\sqrt{\frac{n}{h}}\sum_{i=1}^{n}K_{i-1}^{\star}E_{i-1}[\Delta_{i}Z\Delta_{i}H]\\
& \leq &
2\sqrt{\frac{n}{h}}\sum_{i=1}^{n}K_{i-1}^{\star}\sqrt{E_{i-1}[(\Delta_{i}Z)^{2}]}\sqrt{E_{i-1}[(\Delta_{i}W)^{2}]}\\
& = &
O_{a.s.}\big(ln\frac{1}{\delta}\sqrt{h}\cdot\frac{1}{h}\sum_{i=1}^{n}K_{i-1}^{\star}\delta\big)\\
& = & O_{a.s.}\big(ln\frac{1}{\delta}\sqrt{h}\cdot \frac{K_{1}^{2} -
(K_{1}^{1})^{2}}{\sigma^{4}(x)}L_{X}(T, x)\big) \rightarrow 0.
\end{eqnarray*}
by using the H\"{o}lder inequality.

If H is orthogonal to W, then
\begin{eqnarray*}
\sum_{i=1}^{n}E_{i-1}[q_{i}\Delta_{i}H] & = &
\sqrt{\frac{n}{h}}\sum_{i=1}^{n}K_{i-1}^{\star}E_{i-1}[\Delta_{i}Z\Delta_{i}H]\\
& = &
\sqrt{\frac{n}{h}}\sum_{i=1}^{n}K_{i-1}^{\star}E_{i-1}\big[\int_{t_{i-1}}^{t_{i}}{(Y_{s}
- Y_{i-1})\mu_{s}ds} \Delta_{i}H\big]\\
& = &
O_{a.s.}\big(ln\frac{1}{\delta}\sqrt{h}\cdot\frac{1}{h}\sum_{i=1}^{n}K_{i-1}^{\star}\delta\big)\\
& = & O_{a.s.}\big(ln\frac{1}{\delta}\sqrt{h}\cdot \frac{K_{1}^{2} -
(K_{1}^{1})^{2}}{\sigma^{4}(x)}L_{X}(T, x)\big) \rightarrow 0,
\end{eqnarray*}
provided the boundness of H such that $\Delta_{i}H \leq C.$

\noindent {\textbf{Second Step:} the asymptotic bias for the
numerator of the estimator.

\noindent We now prove the following three results for
\begin{equation}
\label{bias}
\sqrt{\frac{n}{h}}\sum_{i=1}^{n}K(\frac{X_{t_{i-1}}-x}{h})\{\frac{\delta
S_{n,2}}{h^2}-(\frac{X_{t_{i-1}}-x}{h})\frac{\delta
S_{n,1}}{h}\}\int_{t_{i-1}}^{t_{i}}{(\sigma^{2}_{s} -
\sigma^{2}(x))}ds,
\end{equation}
 that is,
\begin{align*}
& A_{1_{n,T}} := \sqrt{\frac{n}{h}}
\sum_{i=1}^{n}K(\frac{X_{t_{i-1}}-x}{h})\int_{t_{i-1}}^{t_{i}}{(\sigma^{2}_{s}
- \sigma^{2}_{i-1})}~ds ~\frac{\delta S_{n,2}}{h^2} =
o_{a.s.}(A_{2_{n,T}}),\\
& B_{1_{n,T}} := \sqrt{\frac{n}{h}}
\sum_{i=1}^{n}K(\frac{X_{t_{i-1}}-x}{h})(\frac{X_{t_{i-1}}-x}{h})\int_{t_{i-1}}^{t_{i}}{(\sigma^{2}_{s}
- \sigma^{2}_{i-1})}~ds ~\frac{\delta S_{n,1}}{h} =
o_{a.s.}(B_{2_{n,T}}),\\
& C_{n,T} := \frac{1}{\sqrt{nh}}\Big(A_{2_{n,T}} + B_{2_{n,T}}\Big)
\stackrel{\mathrm{a.s.}}{\longrightarrow}
\frac{1}{2}(\sigma^{2})^{''}(x)[(K_{1}^{2})^{2} -
K_{1}^{1}K_{1}^{3}] \cdot h^{2},
\end{align*}
where
\begin{align*}
& A_{2_{n,T}} := \sqrt{\frac{n}{h}}
\sum_{i=1}^{n}K(\frac{X_{t_{i-1}}-x}{h})\int_{t_{i-1}}^{t_{i}}{(\sigma^{2}_{i-1}
- \sigma^{2}(x))}~ds ~\frac{\delta S_{n,2}}{h^2},\\
& B_{2_{n,T}} := \sqrt{\frac{n}{h}}
\sum_{i=1}^{n}K(\frac{X_{t_{i-1}}-x}{h})(\frac{X_{t_{i-1}}-x}{h})\int_{t_{i-1}}^{t_{i}}{(\sigma^{2}_{i-1}
- \sigma^{2}(x))}~ds ~\frac{\delta S_{n,1}}{h}.
\end{align*}
\noindent Firstly,
$$\frac{A_{1_{n,T}}}{A_{2_{n,T}}} = \frac{\frac{1}{h}\sum_{i=1}^{n}
K(\frac{X_{t_{i-1}}-x}{h})\int_{t_{i-1}}^{t_{i}}{(\sigma^{2}_{s} -
\sigma^{2}_{i-1})}~ds}{\frac{1}{h}\sum_{i=1}^{n}K(\frac{X_{t_{i-1}}-x}{h})\int_{t_{i-1}}^{t_{i}}{(\sigma^{2}_{i-1}
- \sigma^{2}(x))}~ds}. $$ \noindent By the Taylor expansion for
$\sigma^{2}_{i-1} - \sigma^{2}(x)$ in $A_{2_{n,T}}$ up to order 2,
$$\sigma^{2}_{i-1}
- \sigma^{2}(x) = (\sigma^{2})^{'}(x)(X_{(i-1)\delta} - x) +
\frac{1}{2}(\sigma^{2})^{''}(x + \theta(X_{(i-1)\delta} -
x))(X_{(i-1)\delta} - x)^{2}, $$ \noindent where $\theta$
is a random variable satisfying $\theta \in [0, 1].$

\noindent For $A_{2_{n,T}}$ , by Lemma \ref{l3},
\begin{eqnarray*}
& ~ &
\frac{1}{h}\sum_{i=1}^{n}K(\frac{X_{t_{i-1}}-x}{h})\int_{t_{i-1}}^{t_{i}}{(\sigma^{2}_{i-1}
- \sigma^{2}(x))}~ds\\ & = &
\frac{1}{h}\sum_{i=1}^{n}K(\frac{X_{t_{i-1}}-x}{h})(X_{(i-1)\delta}
- x)\delta~ds*(\sigma^{2})^{'}(x)\\
& + &
\frac{1}{h}\sum_{i=1}^{n}K(\frac{X_{t_{i-1}}-x}{h})(X_{(i-1)\delta}
- x)^{2}\delta~ds*\frac{1}{2}(\sigma^{2})^{''}(x +
\theta(X_{(i-1)\delta} - x))\\
& \stackrel{\mathrm{a.s.}}{\longrightarrow} &
h\frac{(\sigma^{2})^{'}(x)K^{1}_{1}L_{X}(T,x)}{\sigma^{2}(x)} +
\frac{1}{2}h^{2}\frac{(\sigma^{2})^{''}(x)K^{2}_{1}L_{X}(T,x)}{\sigma^{2}(x)}.
\end{eqnarray*}

\noindent Furthermore, we use the mean-value theorem to $\sigma^{2}_{s}
- \sigma^{2}_{i-1}$ for $A_{1_{n,T}},$ then
\begin{eqnarray*}
A_{1_{n,T}} & = &
\frac{1}{h}\sum_{i=1}^{n}K(\frac{X_{t_{i-1}}-x}{h})\int_{t_{i-1}}^{t_{i}}{(\sigma^{2}_{s}
- \sigma^{2}_{i-1})}~ds\\
& = &
\frac{1}{h}\sum_{i=1}^{n}K(\frac{X_{t_{i-1}}-x}{h})\int_{t_{i-1}}^{t_{i}}{(\sigma^{2})^{'}(\xi_{i})(X_{s}
- X_{i-1})}~ds\\
& \stackrel{\mathrm{a.s.}}{\leq} & (\delta ln\frac{1}{\delta})^{\frac{1}{2}}*\sup_{x}|(\sigma^{2})^{'}(x)|
*\frac{1}{h}\sum_{i \in I_{0,n}}K(\frac{X_{t_{i-1}}-x}{h})\delta + 2CN_{1}\delta~(for~i \in I_{1,n})\\
& \rightarrow & O\Big[(\delta
ln\frac{1}{\delta})^{\frac{1}{2}}\frac{L_{X}(T,x)}{\sigma^{2}(x)}\Big]
= o(h)
\end{eqnarray*}
\noindent by the UBI property of $i \in I_{0,n}$.

\noindent Result about $B_{1_{n,T}}$ can be obtained using
$K(u) \cdot u$ instead of $K(u)$ similarly as $A_{1_{n,T}}.$

\noindent Under a simple calculus,
$$ \sum_{i=1}^{n}K(\frac{X_{t_{i-1}}-x}{h})\{\frac{\delta
S_{n,2}}{h^2}-(\frac{X_{t_{i-1}}-x}{h})\frac{\delta S_{n,1}}{h}\}
\cdot (X_{t_{i-1}}-x) = 0.
$$

\noindent For $C_{n,T}$ using the Taylor expansion for
$\sigma^{2}_{i-1} - \sigma^{2}(x)$ up to order 2, we have
\begin{eqnarray*}
C_{n,T} & = &
\frac{1}{h}\sum_{i=1}^{n}K(\frac{X_{t_{i-1}}-x}{h})\{\frac{\delta
S_{n,2}}{h^2}-(\frac{X_{t_{i-1}}-x}{h})\frac{\delta
S_{n,1}}{h}\}\delta\cdot{(\sigma^{2}_{i-1} - \sigma^{2}(x))}\\
& = &
\frac{1}{h}\sum_{i=1}^{n}K(\frac{X_{t_{i-1}}-x}{h})\{\frac{\delta
S_{n,2}}{h^2}-(\frac{X_{t_{i-1}}-x}{h})\frac{\delta
S_{n,1}}{h}\}\delta\\
& \times & [(\sigma^{2})^{'}(x)(X_{(i-1)\delta} - x) +
\frac{1}{2}(\sigma^{2})^{''}(x + \theta(X_{(i-1)\delta} -
x))(X_{(i-1)\delta} - x)^{2}]\\
& = &
\frac{1}{2}\times\frac{1}{h}\sum_{i=1}^{n}K(\frac{X_{t_{i-1}}-x}{h})\frac{\delta
S_{n,2}}{h^2}(\sigma^{2})^{''}(x + \theta(X_{(i-1)\delta} -
x))(X_{(i-1)\delta} - x)^{2}\\
& - &
\frac{1}{2}\times\frac{1}{h}\sum_{i=1}^{n}K(\frac{X_{t_{i-1}}-x}{h})(\frac{X_{t_{i-1}}-x}{h})\frac{\delta
S_{n,1}}{h}(\sigma^{2})^{''}(x + \theta(X_{(i-1)\delta} -
x))(X_{(i-1)\delta} - x)^{2}\\
& =: & C_{1_{n,T}} - C_{2_{n,T}}
\end{eqnarray*}
Similarly as the proof of Lemma \ref{l3},  we obtain
$$C_{1_{n,T}} \stackrel{\mathrm{a.s.}}{\longrightarrow}
\frac{1}{2}(\sigma^{2})^{''}(x)(K_{1}^{2})^{2} \cdot
\Big(\frac{L_{X}(T, x)}{\sigma^{2}(x)}\Big)^{2} \cdot h^{2}$$
$$C_{2_{n,T}} \stackrel{\mathrm{a.s.}}{\longrightarrow}
\frac{1}{2}(\sigma^{2})^{''}(x) K_{1}^{1}K_{1}^{3} \cdot
\Big(\frac{L_{X}(T, x)}{\sigma^{2}(x)}\Big)^{2} \cdot h^{2},$$
\noindent so we have
$$C_{n,T} \stackrel{\mathrm{a.s.}}{\longrightarrow}
\frac{1}{2}(\sigma^{2})^{''}(x)[(K_{1}^{2})^{2} -
K_{1}^{1}K_{1}^{3}] \cdot \Big(\frac{L_{X}(T,
x)}{\sigma^{2}(x)}\Big)^{2} \cdot h^{2}.$$

We complete the proof for Theorem 1.

\subsection{The proof of Theorem \ref{th2}.}
It proceeds basically along the same idea as the detailed
procedure of  Lemma \ref{l3}, which gives the result for $X_{t}$
with finite activity jumps (FA case). As is shown in the proof of
 Lemma \ref{l3}, it is sufficient to prove
\begin{equation}
\label{4.9}
\frac{1}{h}\Big(\sum_{i=1}^{n}
K\big(\frac{X_{t_{i-1}}-x}{h}\big)\big(\frac{X_{t_{i-1}}-x}{h}\big)^{k}\delta
- \int_{0}^{1}{{K\big(\frac{X_{s-} - x}{h}\big)}\big(\frac{X_{s-} -
x}{h}\big)^{k}}ds \Big) \stackrel{p}{\rightarrow} 0 \end{equation}for
$X_{t}$ with finite and infinite activity jumps (IA case). Hence, we
only need to check that the contribution for
$$\frac{1}{h}\sum_{i=1}^{n}K\big(\frac{X_{t_{i-1}}-x}{h}\big)\big(\frac{X_{t_{i-1}}-x}{h}\big)^{k}\delta$$
given by the IA jumps is negligible in the following part based on
the result of Lemma \ref{l3} for FA case.

According to the assumption of $\alpha < 1,$ which means $J$ has
finite variation, we can obtain
$$\tilde{J}_{2t} = \int_{0}^{t}\int_{|x| \leq 1} x~m(dx, ds) - \int_{0}^{t}\int_{|x| \leq 1} x \nu(dx) ds := J_{2t} + C \delta.$$
\noindent Denote $X_{0,t} = \int_{0}^{t} (\mu_{s} + C)ds +
\int_{0}^{t} \sigma_{s} dW_{s},$ we can split $X_{t} = X_{0,t} +
J_{1, t} + J_{2, t}.$ For $i \in I_{1, n} = \{i \in \{1, 2, \cdot
\cdot \cdot n\}: \Delta_{i}N \neq 0\},$ where $N$ is the counting
process with respect to $J_{1, t},$ we have
\begin{eqnarray*}
& ~ & \frac{1}{h}\sum_{i \in
I_{1,n}}\int_{t_{i-1}}^{t_{i}}{\Big(K\big(\frac{X_{t_{i-1}}-x}{h}\big)\cdot\big(\frac{X_{t_{i-1}}
- x}{h}\big)^{k} - {{K\big(\frac{X_{s-} -
x}{h}\big)}\big(\frac{X_{s-} - x}{h}\big)^{k}}\Big)}ds\\ & \leq &
N_{1} \frac{2C\delta}{h} \stackrel{a.s.} \longrightarrow 0.
\end{eqnarray*}

Therefore, we fix $J_{1}$ and regard $K\big(\frac{X_{s-} -
x}{h}\big)\big(\frac{X_{s-} - x}{h}\big)^{k}$ as a two variable
function $F(a, b) := K\Big(\frac{a + J_{1s} + b - x}{h}\Big)
\Big(\frac{a + J_{1s} + b - x}{h}\Big)^{k}$ evaluated at $a =
X_{0,s}$ and $b = J_{2,s}.$

For a function $F(a, b)$ with two variables, by the Taylor
expansion, we have
\begin{eqnarray*}
F(a, b) - F(a_{0}, b_{0}) & = & F_{a}(\xi, \eta)(a - a_{0}) +
F_{b}(\xi, \eta)(b - b_{0})\\
& = & (a - a_{0})\Big[F_{a}(a_{0}, \eta) + F_{aa}(\tilde{\xi},
\eta)(\xi - a_{0})\Big] + F_{b}(\xi, \eta)(b - b_{0})\\
& = & (a - a_{0})\Big[F_{a}(a_{0}, b_{0}) + F_{ab}(a_{0},
\tilde{\eta})(\eta - b_{0}) + F_{aa}(\tilde{\xi}, b_{0})(\xi -
a_{0}) \\& &+ F_{aab}(\tilde{\xi}, \tilde{\eta})(\xi - a_{0})(\eta -
b_{0})\Big] + F_{b}(\xi, \eta)(b - b_{0}),
\end{eqnarray*}
where $F_{a}$ denotes the first partial derivative of the function
$F.$

Using the expansion equation for
$K\big(\frac{X_{t_{i-1}}-x}{h}\big)\cdot\big(\frac{X_{t_{i-1}} -
x}{h}\big)^{k} - {{K\big(\frac{X_{s-} - x}{h}\big)}\big(\frac{X_{s-}
- x}{h}\big)^{k}}$ around $(a_{0}, b_{0}) = (X_{0s}, J_{2s})$ with
$(a, b) = (X_{0, t_{i-1}} J_{2, t_{i-1}})$ and $F( \cdot ) = K(
\cdot ) ( \cdot )^{k}$, we reach
\begin{eqnarray*}
& ~ & \Big|
K\big(\frac{X_{t_{i-1}}-x}{h}\big)\cdot\big(\frac{X_{t_{i-1}} -
x}{h}\big)^{k} - {{K\big(\frac{X_{s-} - x}{h}\big)}\big(\frac{X_{s-}
- x}{h}\big)^{k}} \Big|\\
& := & |F(a, b) - F(a_{0}, b_{0})|\\
& \leq & \frac{|F_{a}(X_{0s}, J_{2s})|}{h_{n}}\sup_{u \in (t_{i-1},
t_{i}]}|X_{0,t_{i-1}} - X_{0u}|\\ & ~ & + \frac{|F_{ab}(X_{0s},
\tilde{J}_{2s})|}{h_{n}^{2}}\sup_{u \in (t_{i-1},
t_{i}]}|X_{0,t_{i-1}} - X_{0u}| \cdot |J_{2, t_{i-1}} - J_{2s}|\\
& ~ & + \frac{|F_{aa}(\tilde{X}_{0s}, {J}_{2s})|}{h_{n}^{2}} \sup_{u
\in (t_{i-1}, t_{i}]}|X_{0,t_{i-1}} - X_{0u}|^{2} \\ & &+
\frac{|F_{aab}(\tilde{X}_{0s}, \tilde{J}_{2s})|}{h_{n}^{3}}\sup_{u
\in (t_{i-1}, t_{i}]}|X_{0,t_{i-1}} - X_{0u}|^{2} \cdot |J_{2,
t_{i-1}} - J_{2s}| + \frac{|F_{b}(\tilde{X}_{0s}, \tilde{J}_{2s})|}{h_{n}} |J_{2,
t_{i-1}} - J_{2s}|,
\end{eqnarray*}
where $\tilde{X}_{0s}, \tilde{J}_{2s}$ are the suitable points to
give the Lagrange remainder for the Taylor expansion and $F_{a}$
denotes the derivative for $K( \cdot ) ( \cdot )^{k}$ with respect
to the first variable for simplicity.

According to the Taylor expansion, we have
\begin{eqnarray*}
& ~ & \Big|\frac{1}{h}\Big(\sum_{i=1}^{n}
K\big(\frac{X_{t_{i-1}}-x}{h}\big)\big(\frac{X_{t_{i-1}}-x}{h}\big)^{k}\delta
- \int_{0}^{1}{{K\big(\frac{X_{s-} - x}{h}\big)}\big(\frac{X_{s-} -
x}{h}\big)^{k}}ds \Big)\Big|\\
& = & \frac{1}{h}\sum_{i =
1}^{n}\int_{t_{i-1}}^{t_{i}}{\Big|K\big(\frac{X_{t_{i-1}}-x}{h}\big)\cdot\big(\frac{X_{t_{i-1}}
- x}{h}\big)^{k} - {{K\big(\frac{X_{s-} -
x}{h}\big)}\big(\frac{X_{s-} - x}{h}\big)^{k}}\Big|}ds\\
& \leq & \frac{1}{h}\sum_{i =
1}^{n}\int_{t_{i-1}}^{t_{i}}\Big[\frac{|F_{a}(X_{0s},
J_{2s})|}{h_{n}}\sup_{u \in (t_{i-1}, t_{i}]}|X_{0,t_{i-1}} -
X_{0u}|\\ & & + \frac{|F_{ab}(X_{0s}, \tilde{J}_{2s})|}{h_{n}^{2}}\sup_{u
\in (t_{i-1}, t_{i}]}|X_{0,t_{i-1}} - X_{0u}| \cdot |J_{2, t_{i-1}}
- J_{2s}|\\
& ~ & + \frac{|F_{aa}(\tilde{X}_{0s}, {J}_{2s})|}{h_{n}^{2}} \sup_{u
\in (t_{i-1}, t_{i}]}|X_{0,t_{i-1}} - X_{0u}|^{2} \\ & &+
\frac{|F_{aab}(\tilde{X}_{0s}, \tilde{J}_{2s})|}{h_{n}^{3}}\sup_{u
\in (t_{i-1}, t_{i}]}|X_{0,t_{i-1}} - X_{0u}|^{2} \cdot |J_{2,
t_{i-1}} - J_{2s}|\\
& ~ & + \frac{|F_{b}(\tilde{X}_{0s}, \tilde{J}_{2s})|}{h_{n}} |J_{2,
t_{i-1}} - J_{2s}|\Big] ds\\
& := & \Pi_{1_{n, T}} + \Pi_{2_{n, T}} + \Pi_{3_{n, T}} + \Pi_{4_{n,
T}} + \Pi_{5_{n, T}}.
\end{eqnarray*}

\noindent We now show that the five terms give a negligible
contribution.

\noindent {\textbf{For $\Pi_{1_{n, T}}$}.

Using the UBI property
of $X_{0}$ and the occupation time formula, we get
\begin{eqnarray*}
& ~ & \frac{1}{h_{n}}\sum_{i = 1}^{n}\int_{t_{i-1}}^{t_{i}}
\frac{|F_{a}(X_{0s}, J_{2s})|}{h_{n}}\sup_{u \in (t_{i-1},
t_{i}]}|X_{0,t_{i-1}} - X_{0u}| ds\\
& \leq & C \frac{(\delta \ln(1/\delta))^{1/2}}{h_{n}}
\frac{1}{h_{n}} \int_{0}^{1}{|F_{a}(X_{0s}, J_{2s})|}ds\\
& = & C \frac{(\delta \ln(1/\delta))^{1/2}}{h_{n}}
\int_{\mathbb{R}}{|F_{a}(u)|\frac{L(u + h_{n}x)}{\sigma^{2}(u +
h_{n}x)}}du \stackrel{a.s.} \longrightarrow 0.
\end{eqnarray*}

\noindent {\textbf{For $\Pi_{2_{n, T}}$}.

Using the UBI property
of $X_{0}$, the boundedness of $F_{ab}(\cdot, \cdot)$ and $E
[|\tilde{J}_{2, s} - \tilde{J}_{2, t_{i-1}}|^{2}] = O (\delta)$ with
H\"{o}lder inequality for $s \in [t_{i - 1}, t_{i}]$ (one can refer
to this equation in the Proof of Theorem 4 for Mancini and Ren\`{o}
\cite{mr}), we have
\begin{eqnarray*}
& ~ & E \Big[\frac{1}{h_{n}}\sum_{i =
1}^{n}\int_{t_{i-1}}^{t_{i}}\frac{|F_{ab}(X_{0s},
\tilde{J}_{2s})|}{h_{n}^{2}}\sup_{u \in (t_{i-1},
t_{i}]}|X_{0,t_{i-1}} - X_{0u}| \cdot |J_{2, t_{i-1}} - J_{2s}| ds \Big]\\
& \leq & C E \Big[ \frac{1}{h_{n}}\sum_{i =
1}^{n}\int_{t_{i-1}}^{t_{i}}\frac{1}{h_{n}^{2}} (\delta
\ln(1/\delta))^{1/2} \cdot
|J_{2, t_{i-1}} - J_{2s}| ds \Big]\\
& = & C \frac{(\delta \ln(1/\delta))^{1/2}}{h_{n}^{2}}
\frac{1}{h_{n}}\sum_{i = 1}^{n}\int_{t_{i-1}}^{t_{i}} E [|J_{2,
t_{i-1}} - J_{2s}|] ds\\
& = & C \frac{(\delta \ln(1/\delta))^{1/2}}{h_{n}^{2}}
\frac{O(\sqrt{\delta})}{h_{n}} = O(1) \frac{\delta }{h_{n}^{3}}
(\ln(1/\delta))^{1/2} \longrightarrow 0,
\end{eqnarray*}
hence, it is shown that $\Pi_{2_{n, T}} \stackrel{p} \longrightarrow
0.$

\noindent {\textbf{For $\Pi_{3_{n, T}}$}.

\noindent Using the UBI property of $X_{0}$ and the occupation time
formula, we have
\begin{eqnarray*}
& ~ & \frac{1}{h_{n}}\sum_{i = 1}^{n}\int_{t_{i-1}}^{t_{i}}
\frac{|F_{aa}(\tilde{X}_{0s}, {J}_{2s})|}{h_{n}^{2}} \sup_{u \in
(t_{i-1}, t_{i}]}|X_{0,t_{i-1}} - X_{0u}|^{2}ds\\
& \leq & C \frac{\delta \ln(1/\delta)}{h_{n}^{2}}
\frac{1}{h_{n}}\sum_{i = 1}^{n}\int_{t_{i-1}}^{t_{i}}
|F_{aa}(\tilde{X}_{0s}, {J}_{2s})| ds\\
& = & C \frac{\delta \ln(1/\delta)}{h_{n}^{2}}
\int_{\mathbb{R}}{|F_{aa}(u)|\frac{L(u + h_{n}x)}{\sigma^{2}(u +
h_{n}x)}}du \stackrel{a.s.} \longrightarrow 0.
\end{eqnarray*}

\noindent {\textbf{For $\Pi_{4_{n, T}}$}.

Using the UBI property
of $X_{0}$, the boundedness of $F_{aab}(\cdot, \cdot)$ and $E
[|\tilde{J}_{2, s} - \tilde{J}_{2, t_{i-1}}|^{2}] = O(\delta)$ with
H\"{o}lder inequality for $s \in [t_{i - 1}, t_{i}]$, it can be
shown
\begin{eqnarray*}
& ~ & E \Big[\frac{1}{h_{n}}\sum_{i = 1}^{n}\int_{t_{i-1}}^{t_{i}}
\frac{|F_{aab}(\tilde{X}_{0s}, \tilde{J}_{2s})|}{h_{n}^{3}}\sup_{u
\in (t_{i-1}, t_{i}]}|X_{0,t_{i-1}} - X_{0u}|^{2} \cdot |J_{2,
t_{i-1}} - J_{2s}| ds \Big]\\
& \leq & C E\Big[ \frac{1}{h_{n}}\sum_{i =
1}^{n}\int_{t_{i-1}}^{t_{i}} \frac{1}{h_{n}^{3}} \cdot \delta
\ln(1/\delta) \cdot |J_{2, t_{i-1}} - J_{2s}| ds \Big]\\
& = & C \frac{\delta \ln(1/\delta)}{h_{n}^{3}}
\frac{1}{h_{n}}\sum_{i = 1}^{n}\int_{t_{i-1}}^{t_{i}} E [|J_{2,
t_{i-1}} - J_{2s}|] ds\\
& = & C \frac{\delta \ln(1/\delta)}{h_{n}^{3}}
\frac{\sqrt{\delta}}{h_{n}} = O(1) \frac{(\delta)^{3/2}}{h_{n}^{4}}
\ln(1/\delta) \longrightarrow 0,
\end{eqnarray*}
so we have prove $\Pi_{4_{n, T}} \stackrel{p} \longrightarrow 0.$

\noindent {\textbf{For $\Pi_{5_{n, T}}$}.

Using the boundedness
of $F_{b}(\cdot, \cdot)$ and $E [|\tilde{J}_{2, s} - \tilde{J}_{2,
t_{i-1}}|^{2}] = O(\delta)$ with H\"{o}lder inequality for $s \in
[t_{i - 1}, t_{i}]$, we can prove
\begin{eqnarray*}
& ~ & E \Big[ \frac{1}{h_{n}}\sum_{i = 1}^{n}\int_{t_{i-1}}^{t_{i}}
\frac{|F_{b}(\tilde{X}_{0s}, \tilde{J}_{2s})|}{h_{n}} |J_{2,
t_{i-1}} - J_{2s}| ds \Big]\\
& \leq & C \frac{1}{h_{n}^{2}} \sum_{i =
1}^{n}\int_{t_{i-1}}^{t_{i}} E
[|J_{2, t_{i-1}} - J_{2s}|] ds\\
& = & O(1) \frac{\sqrt{\delta}}{h_{n}^{2}} \longrightarrow 0,
\end{eqnarray*}
so we have $\Pi_{5_{n, T}} \stackrel{p} \longrightarrow 0.$

From the above five parts, we get
$$\frac{1}{h}\Big(\sum_{i=1}^{n}
K\big(\frac{X_{t_{i-1}}-x}{h}\big)\big(\frac{X_{t_{i-1}}-x}{h}\big)^{k}\delta
- \int_{0}^{1}{{K\big(\frac{X_{s-} - x}{h}\big)}\big(\frac{X_{s-} -
x}{h}\big)^{k}}ds \Big) \stackrel{p}{\rightarrow} 0,$$ so we have
$$\frac{1}{h}\sum_{i=1}^{n}K\big(\frac{X_{t_{i-1}}-x}{h}\big)\big(\frac{X_{t_{i-1}}-x}{h}\big)^{k}\delta~
\stackrel{p}{\longrightarrow}~\frac{K_{1}^{k}L_{X}(T,x)}{\sigma^2(x)}$$
with the result $\frac{1}{h}\int_{0}^{1}{{K\big(\frac{X_{s-} -
x}{h}\big)}\big(\frac{X_{s-} - x}{h}\big)^{k}}ds \stackrel{a.s.}
\longrightarrow \frac{K_{1}^{k}L_{X}(T,x)}{\sigma^2(x)}$ in the
detailed proof of Lemma \ref{l3}
\bigskip

For the consistency and asymptotic normality for
$\hat{\sigma}^{2}_{n}(x),$ we follow the same procedures as that in
the detailed proof of Theorem \ref{th1}, it is sufficient to check
that the contribution given by the IA jumps $\tilde{J}_{2}$ is
negligible at each step based on the result of Theorem \ref{th1} for
FA case. Write
\begin{eqnarray*}
& ~ & \sqrt{nh}(\hat{\sigma}_{n}^{2}(x) - \sigma^{2}(x))\\ & = &
\frac{\sqrt{\frac{n}{h}}\sum_{i=1}^{n}K(\frac{X_{t_{i-1}}-x}{h})\{\frac{\delta
S_{n,2}}{h^2}-(\frac{X_{t_{i-1}}-x}{h})\frac{\delta
S_{n,1}}{h}\}\big((\Delta_{i}X)^{2} I_{\{(\Delta_{i}X)^{2}
 \leq \vartheta(\delta_{k})\}} -
\sigma^{2}(x)\delta\big)}{\frac{1}{h}\sum_{i=1}^{n}K(\frac{X_{t_{i-1}}-x}{h})\{\frac{\delta
S_{n,2}}{h^2}-(\frac{X_{t_{i-1}}-x}{h})\frac{\delta
S_{n,1}}{h}\}\delta}.
\end{eqnarray*}
To prove the result it is sufficient to show that
the numerator tends stably in law to random variable $M_{1}.$

Recall the following fact,  if $Z_{n} \stackrel{\mathcal {S} - \mathcal {L}}
\longrightarrow Z $ and if $Y_{n}$ and $Y$ are variables defined on
$(\Omega, \mathcal {F}, \mathbb{P})$ and with values in the same
Polish space F, then
\begin{equation}
\label{4.10}
Y_{n} \stackrel{P}
\longrightarrow Y~~~~~\Rightarrow~~~~~(Y_{n}, ~Z_{n})
\stackrel{\mathcal {S} - \mathcal {L}} \longrightarrow (Y, ~Z),$$
which implies that $Y_{n} + Z_{n} \stackrel{\mathcal {S} - \mathcal
{L}} \longrightarrow Y + Z.$ Hence, we need to prove that
$$\sqrt{\frac{n}{h}}\sum_{i=1}^{n}K(\frac{X_{t_{i-1}}-x}{h})\{\frac{\delta
S_{n,2}}{h^2}-(\frac{X_{t_{i-1}}-x}{h})\frac{\delta
S_{n,1}}{h}\}\Big((\Delta_{i}X)^{2}
I_{\{(\Delta_{i}\tilde{J_{2}})^{2} \leq
4\vartheta(\delta_{k}),~\Delta_{i}N = 0\}} -
\sigma^{2}(x)\delta\Big) \stackrel{s.t.} \longrightarrow M_{1}.\end{equation}

\noindent For (\ref{4.10}), we have
\begin{eqnarray*}
& ~ &
\sqrt{\frac{n}{h}}\sum_{i=1}^{n}K(\frac{X_{t_{i-1}}-x}{h})\{\frac{\delta
S_{n,2}}{h^2}-(\frac{X_{t_{i-1}}-x}{h})\frac{\delta
S_{n,1}}{h}\}\Big((\Delta_{i}X)^{2}
I_{\{(\Delta_{i}\tilde{J_{2}})^{2} \leq
4\vartheta(\delta_{k}),~\Delta_{i}N = 0\}} -
\sigma^{2}(x)\delta\Big)\\
& = &
\sqrt{\frac{n}{h}}\sum_{i=1}^{n}K(\frac{X_{t_{i-1}}-x}{h})\{\frac{\delta
S_{n,2}}{h^2}-(\frac{X_{t_{i-1}}-x}{h})\frac{\delta S_{n,1}}{h}\}
\times \\
& ~ & \times \Big((\Delta_{i}Y + \Delta_{i}\tilde{J}_{2})^{2}
I_{\{(\Delta_{i}\tilde{J_{2}})^{2} \leq
4\vartheta(\delta_{k}),~\Delta_{i}N = 0\}} -
\sigma^{2}(x)\delta\Big)\\
& = &
\sqrt{\frac{n}{h}}\sum_{i=1}^{n}K(\frac{X_{t_{i-1}}-x}{h})\{\frac{\delta
S_{n,2}}{h^2}-(\frac{X_{t_{i-1}}-x}{h})\frac{\delta
S_{n,1}}{h}\}\Big((\Delta_{i}Y)^{2} - \sigma^{2}(x)\delta\Big)\\
& ~ & -
\sqrt{\frac{n}{h}}\sum_{i=1}^{n}K(\frac{X_{t_{i-1}}-x}{h})\{\frac{\delta
S_{n,2}}{h^2}-(\frac{X_{t_{i-1}}-x}{h})\frac{\delta
S_{n,1}}{h}\}(\Delta_{i}Y)^{2} I_{\{(\Delta_{i}\tilde{J_{2}})^{2}
> 4\vartheta(\delta_{k})\} \bigcup \{\Delta_{i}N \neq 0\}}\\
& ~ & + 2
\sqrt{\frac{n}{h}}\sum_{i=1}^{n}K(\frac{X_{t_{i-1}}-x}{h})\{\frac{\delta
S_{n,2}}{h^2}-(\frac{X_{t_{i-1}}-x}{h})\frac{\delta S_{n,1}}{h}\}
\Delta_{i}Y \Delta_{i}\tilde{J}_{2}
I_{\{(\Delta_{i}\tilde{J_{2}})^{2} \leq
4\vartheta(\delta_{k}),~\Delta_{i}N = 0\}}\\
& ~ & +
\sqrt{\frac{n}{h}}\sum_{i=1}^{n}K(\frac{X_{t_{i-1}}-x}{h})\{\frac{\delta
S_{n,2}}{h^2}-(\frac{X_{t_{i-1}}-x}{h})\frac{\delta S_{n,1}}{h}\}
(\Delta_{i}\tilde{J}_{2})^{2} I_{\{(\Delta_{i}\tilde{J_{2}})^{2}
\leq 4\vartheta(\delta_{k}),~\Delta_{i}N = 0\}}\\
& := & \sum_{l = 1}^{4} \mathcal {S}_{l}.
\end{eqnarray*}
From the detailed proof for $X_{t}$ with finite activity jump (FA
case) in Theorem \ref{th1}, we have shown that
$$\mathcal
{S}_{1} - \sqrt{nh_{n}} \cdot
\frac{1}{2}(\sigma^{2})^{''}(x)[(K_{1}^{2})^{2} -
K_{1}^{1}K_{1}^{3}] \cdot h_{n}^{2} \stackrel{\mathcal {S} -
\mathcal {L}} \longrightarrow M_{1},$$ where $M_{1}$ is a Gaussian
martingale defined on an extension $\big(\tilde{\Omega} , \tilde{P}
, \tilde{\mathscr{F}}\big)$ of our filtered probability space and
having $\tilde{E}[M_{1}^{2}|\mathscr{F}] =
\frac{2}{\sigma^{2}(x)}\cdot L_{X}^{3}(T,x)\cdot V_{x}^{'}$ with
$V_{x} = K_{2}^{0}\cdot\big(K_{1}^{2}\big)^{2} +
K_{2}^{2}\cdot\big(K_{1}^{1}\big)^{2} - 2K_{2}^{1}\cdot
K_{1}^{2}\cdot K_{1}^{1}.$
\medskip

In the following part, we will verify the fact that $\mathcal
{S}_{1}$ converges stably to $M_{1}$, the other terms tend to zero
in probability for $X_{t}$ with infinite activity jump (IA case)
similarly as the result in Theorem \ref{th1}, that is, the contribution
given by infinite activity jumps can be negligible.
\medskip

\noindent {\textbf{For $\mathcal {S}_{1}$}.

It consists of two terms $$\sum_{i=1}^{n}q_{i} +
\sqrt{\frac{n}{h}}\sum_{i=1}^{n}K(\frac{X_{t_{i-1}}-x}{h})\{\frac{\delta
S_{n,2}}{h^2}-(\frac{X_{t_{i-1}}-x}{h})\frac{\delta
S_{n,1}}{h}\}\int_{t_{i-1}}^{t_{i}}{(\sigma^{2}_{s} -
\sigma^{2}(x))}ds.$$

For the first term $\sum_{i = 1}^{n}q_{i} = L_{T}(T, x)\sum_{i =
1}^{n}q_{i}^{'}$, $\sum_{i = 1}^{n}q_{i}^{'}$ is composed of four
parts such as $S_{1},~S_{2},~S_{3},~S_{4}$.

$S_{1},~S_{3},~S_{4}$ these three parts can be dealt as the FA jumps
case in Theorem \ref{th1}. For $S_{2},$ in $D1,~D3,$ we similarly
expand $\sigma^{2}_{s},~ \sigma^{4}_{s}$ up to the first order
respectively. Using the H\"{o}lder, BDG inequality and the IA jump
component contribution with $E [|\tilde{J}_{2, s} - \tilde{J}_{2,
t_{i-1}}|^{2}] = O(\delta)$, we can obtain the convergence of
$D1,~D3$ to 0 in probability. $D2,~D5$ can be proved with the
similar procedure as that in Theorem \ref{th1} using the occupation
time formula. $D3$ can be dealt by the similar steps.

\noindent {\textbf{For $\mathcal {S}_{2}$}.

The sum of the terms with $\Delta_{i}N \neq 0$ is
$O_{p}(\sqrt{\frac{\delta}{h_{n}}}) \rightarrow 0$ by  Lemma \ref{l2}.
For the sum of the terms with $(\Delta_{i}\tilde{J_{2}})^{2}
> 4\vartheta(\delta_{k}),$
 we consider
\begin{eqnarray*}
\frac{\mathcal {S}_{2}}{L_{X}(T, x)} & \stackrel{p} \longrightarrow
& \sqrt{\frac{n}{h}}\sum_{i = 1}^{n}K(\frac{X_{t_{i-1}}-x}{h})\Big\{
\frac{K_{1}^{2}}{\sigma^{2}(x)} - \frac{X_{t_{i-1}}-x}{h} \cdot
\frac{K_{1}^{1}}{\sigma^{2}(x)}\Big\} (\Delta_{i}Y)^{2}
I_{(\Delta_{i}\tilde{J_{2}})^{2}
> 4\vartheta(\delta_{k})}\\ & := & \mathcal {S}_{2}^{'}.
\end{eqnarray*}

\noindent For $\mathcal {S}_{2}^{'}$, we have
\begin{eqnarray*}
& &E_{i - 1}[|\mathcal {S}_{2}^{'}|] \\& \leq & \sqrt{\frac{n}{h}}\sum_{i
= 1}^{n}K(\frac{X_{t_{i-1}}-x}{h})\Big\{
\frac{K_{1}^{2}}{\sigma^{2}(x)} + sign(K_{1}^{1})
\frac{X_{t_{i-1}}-x}{h} \cdot \frac{K_{1}^{1}}{\sigma^{2}(x)}\Big\}
E_{i - 1} \big[(\Delta_{i}Y)^{2} I_{(\Delta_{i}\tilde{J_{2}})^{2}
> 4\vartheta(\delta_{k})}\big]\\
& \leq & \sqrt{\frac{n}{h}}\sum_{i =
1}^{n}K(\frac{X_{t_{i-1}}-x}{h})\Big\{
\frac{K_{1}^{2}}{\sigma^{2}(x)} + sign(K_{1}^{1})
\frac{X_{t_{i-1}}-x}{h} \cdot \frac{K_{1}^{1}}{\sigma^{2}(x)}\Big\}
\times \\
& ~ & \times E_{i - 1}^{1/p} \big[(\Delta_{i}Y)^{2p} \big] P_{i -
1}^{1/q}\{(\Delta_{i}\tilde{J_{2}})^{2}
> 4\vartheta(\delta_{k})\}\\
& \leq & \delta^{- 1/2 + \phi/2} \cdot \frac{1}{h} \sum_{i =
1}^{n}K(\frac{X_{t_{i-1}}-x}{h})\Big\{
\frac{K_{1}^{2}}{\sigma^{2}(x)} + sign(K_{1}^{1})
\frac{X_{t_{i-1}}-x}{h} \cdot \frac{K_{1}^{1}}{\sigma^{2}(x)}\Big\}
C_{p} \delta \cdot \delta^{(1 - \alpha\eta)1/q}\\
& = & C \delta^{(1 - \alpha\eta)1/q - 1/2 + \phi/2} \cdot
\frac{1}{h} \sum_{i = 1}^{n}K(\frac{X_{t_{i-1}}-x}{h})\Big\{
\frac{K_{1}^{2}}{\sigma^{2}(x)} + sign(K_{1}^{1})
\frac{X_{t_{i-1}}-x}{h} \cdot \frac{K_{1}^{1}}{\sigma^{2}(x)}\Big\}
\delta\\
& \rightarrow & C \delta^{(1 - \alpha\eta)1/q - 1/2 + \phi/2} \cdot
\Big[\frac{L_{X}(T, x)}{\sigma^{2}(x)} \cdot
\frac{K_{1}^{2}}{\sigma^{2}(x)} + sign(K_{1}^{1}) \cdot
\frac{L_{X}(T, x) K_{1}^{2}}{\sigma^{2}(x)} \cdot
\frac{K_{1}^{1}}{\sigma^{2}(x)}\Big]~(in~probability )\\
& = & O\big(\delta^{(1 - \alpha\eta)1/q - 1/2 + \phi/2}\big)
\longrightarrow 0,
\end{eqnarray*}
using the H\"{o}lder inequality with $q$ close to 1, the BDG
inequality. Hence we prove the convergence of
$\mathcal {S}_{2}$ to 0 in probability.

\noindent {\textbf{For $\mathcal {S}_{3}$},

We consider
\begin{eqnarray*}
\frac{\mathcal {S}_{3}}{L_{X}(T, x)} & \stackrel{p} \longrightarrow
& \sqrt{\frac{n}{h}}\sum_{i = 1}^{n}K(\frac{X_{t_{i-1}}-x}{h})\Big\{
\frac{K_{1}^{2}}{\sigma^{2}(x)} - \frac{X_{t_{i-1}}-x}{h} \cdot
\frac{K_{1}^{1}}{\sigma^{2}(x)}\Big\} \Delta_{i}Y
\Delta_{i}\tilde{J}_{2} I_{\{(\Delta_{i}\tilde{J_{2}})^{2} \leq
4\vartheta(\delta_{k}),~\Delta_{i}N = 0\}}\\
& := & \mathcal{S}_{3}^{'}.
\end{eqnarray*}

\noindent For $\mathcal {S}_{3}^{'}$, we have
\begin{eqnarray*}
& ~ & E_{i - 1}[|\mathcal {S}_{3}^{'}|]\\
 & \leq &
\sqrt{\frac{n}{h}}\sum_{i = 1}^{n}K(\frac{X_{t_{i-1}}-x}{h})\Big\{
\frac{K_{1}^{2}}{\sigma^{2}(x)} + sign(K_{1}^{1})
\frac{X_{t_{i-1}}-x}{h} \cdot \frac{K_{1}^{1}}{\sigma^{2}(x)}\Big\}
E_{i - 1} \big[ |\Delta_{i}Y \Delta_{i}\tilde{J}_{2}|
I_{\{(\Delta_{i}\tilde{J_{2}})^{2} \leq 4\vartheta(\delta_{k})\}}
\big]\\
& \leq & \sqrt{\frac{n}{h}}\sum_{i =
1}^{n}K(\frac{X_{t_{i-1}}-x}{h})\Big\{
\frac{K_{1}^{2}}{\sigma^{2}(x)} + sign(K_{1}^{1})
\frac{X_{t_{i-1}}-x}{h} \cdot \frac{K_{1}^{1}}{\sigma^{2}(x)}\Big\}
\times \\
& ~ & \times E_{i - 1}^{1/p} \big[ (\Delta_{i}Y)^{p} \big] E_{i -
1}^{1/q} \big[
(\Delta_{i}\tilde{J}_{2})^{q}I_{\{(\Delta_{i}\tilde{J_{2}})^{2} \leq
4\vartheta(\delta_{k})\}} \big]\\
& \leq & \delta^{- 1/2 + \phi/2} \cdot \frac{1}{h} \sum_{i =
1}^{n}K(\frac{X_{t_{i-1}}-x}{h})\Big\{
\frac{K_{1}^{2}}{\sigma^{2}(x)} + sign(K_{1}^{1})
\frac{X_{t_{i-1}}-x}{h} \cdot \frac{K_{1}^{1}}{\sigma^{2}(x)}\Big\}
\cdot \delta^{1/p} \cdot \delta^{1/q + \eta/q(1 -
\frac{\alpha}{2})}\\
& = & \delta^{- 1/2 + \phi/2 + \eta/q(1 - \frac{\alpha}{2})} \cdot
\frac{1}{h} \sum_{i = 1}^{n}K(\frac{X_{t_{i-1}}-x}{h})\Big\{
\frac{K_{1}^{2}}{\sigma^{2}(x)} + sign(K_{1}^{1})
\frac{X_{t_{i-1}}-x}{h} \cdot
\frac{K_{1}^{1}}{\sigma^{2}(x)}\Big\}\delta\\
 & \rightarrow &
\delta^{- 1/2 + \phi/2 + \eta/q(1 - \frac{\alpha}{2})} \cdot
\Big[\frac{L_{X}(T, x)}{\sigma^{2}(x)} \cdot
\frac{K_{1}^{2}}{\sigma^{2}(x)} + sign(K_{1}^{1}) \cdot
\frac{L_{X}(T, x) K_{1}^{2}}{\sigma^{2}(x)} \cdot
\frac{K_{1}^{1}}{\sigma^{2}(x)}\Big]\\
& = & O\big(\delta^{- 1/2 + \phi/2 + \eta/q(1 -
\frac{\alpha}{2})}\big) \longrightarrow 0,
\end{eqnarray*}
using H\"{o}lder inequality with $q$ close to 1 andthe BDG inequality.
Hence,
we prove the convergence of $\mathcal {S}_{3}$ to 0 in probability.

\noindent {\textbf{For $\mathcal {S}_{4}$}, We consider
\begin{eqnarray*}
\frac{\mathcal {S}_{4}}{L_{X}(T, x)} & \stackrel{p} \longrightarrow
& \sqrt{\frac{n}{h}}\sum_{i = 1}^{n}K(\frac{X_{t_{i-1}}-x}{h})\Big\{
\frac{K_{1}^{2}}{\sigma^{2}(x)} - \frac{X_{t_{i-1}}-x}{h} \cdot
\frac{K_{1}^{1}}{\sigma^{2}(x)}\Big\} (\Delta_{i}\tilde{J}_{2})^{2}
I_{\{(\Delta_{i}\tilde{J_{2}})^{2} \leq
4\vartheta(\delta_{k}),~\Delta_{i}N = 0\}}\\ & := & \mathcal
{S}_{4}^{'}.
\end{eqnarray*}

\noindent For $\mathcal {S}_{4}^{'}$, we have
\begin{eqnarray*}
E_{i - 1}[|\mathcal {S}_{4}^{'}|] & \leq & \sqrt{\frac{n}{h}}\sum_{i
= 1}^{n}K(\frac{X_{t_{i-1}}-x}{h})\Big\{
\frac{K_{1}^{2}}{\sigma^{2}(x)} + sign(K_{1}^{1})
\frac{X_{t_{i-1}}-x}{h} \cdot \frac{K_{1}^{1}}{\sigma^{2}(x)}\Big\}
\times \\
& ~ & \times E_{i - 1} \big[ (\Delta_{i}\tilde{J}_{2})^{2}
I_{\{(\Delta_{i}\tilde{J_{2}})^{2} \leq
4\vartheta(\delta_{k})\}} \big]\\
& = & \delta^{- 1/2 + \phi/2} \cdot \frac{1}{h} \sum_{i =
1}^{n}K(\frac{X_{t_{i-1}}-x}{h})\Big\{
\frac{K_{1}^{2}}{\sigma^{2}(x)} + sign(K_{1}^{1})
\frac{X_{t_{i-1}}-x}{h} \cdot \frac{K_{1}^{1}}{\sigma^{2}(x)}\Big\}
\cdot \delta^{1 + \eta(1 - \frac{\alpha}{2})}\\
& = & \delta^{- 1/2 + \phi/2 + \eta(1 - \frac{\alpha}{2})} \cdot
\frac{1}{h} \sum_{i = 1}^{n}K(\frac{X_{t_{i-1}}-x}{h})\Big\{
\frac{K_{1}^{2}}{\sigma^{2}(x)} + sign(K_{1}^{1})
\frac{X_{t_{i-1}}-x}{h} \cdot \frac{K_{1}^{1}}{\sigma^{2}(x)}\Big\}
\delta\\
& \rightarrow & \delta^{- 1/2 + \phi/2 + \eta(1 - \frac{\alpha}{2})}
\cdot \Big[\frac{L_{X}(T, x)}{\sigma^{2}(x)} \cdot
\frac{K_{1}^{2}}{\sigma^{2}(x)} + sign(K_{1}^{1}) \cdot
\frac{L_{X}(T, x) K_{1}^{2}}{\sigma^{2}(x)} \cdot
\frac{K_{1}^{1}}{\sigma^{2}(x)}\Big]\\
& = & O\big(\delta^{- 1/2 + \phi/2 + \eta(1 -
\frac{\alpha}{2})}\big) \longrightarrow 0.
\end{eqnarray*}

Hence we prove the convergence of $\mathcal {S}_{4}$ to 0 in
probability.

We complete the proof for Theorem 2.

\end{document}